\documentclass[12pt]{article}%
\pdfoutput=1
\usepackage{amsmath}
\usepackage{url}
\usepackage{amsfonts}
\usepackage{amssymb}
\usepackage{amsthm}
\usepackage{graphicx}
\usepackage{caption}
\providecommand{\U}[1]{\protect\rule{.1in}{.1in}}

\newtheorem*{untheorem}{Theorem}
\newtheorem*{unremark}{Remark}
\newtheorem*{uncorollary}{Corollary}
\newtheorem*{undefinition}{Definition}

\newtheorem*{EWW}{Every-which-way Theorem}
\newtheorem*{CRT}{Conjugation Rotation Theorem}
\def\co{\colon\thinspace}
\begin{document}

\title{How efficiently can one untangle a double-twist? Waving is believing!}
\author{David Pengelley
\and Daniel Ramras\thanks{The second author was partially supported by a grant from
the Simons Foundation (\# 279007).}}
\date{}
\maketitle

\begin{abstract}
It has long been known to mathematicians and physicists that while a full
rotation in three-dimensional Euclidean space causes tangling, two rotations
can be untangled. Formally, an untangling is a based nullhomotopy of the
double-twist loop in the special orthogonal group of rotations. We study a
particularly simple, geometrically defined untangling procedure, leading to
new conclusions regarding the minimum possible complexity of untanglings. We
animate and analyze how our untangling operates on frames in 3--space, and
teach readers in a video how to wave the nullhomotopy with their hands.

\end{abstract}

\section{Entanglement, efficiency, and aesthetics}

When an object receives a full rotational turn, it becomes \textquotedblleft
entangled\textquotedblright\ with its environment. This is second nature to
anyone who winds up a garden hose. As more turns are applied, we feel
intuitively that the entanglement magnifies via increasingly wound strings
(visible or invisible) connecting the object to fixed surroundings. However,
outlandish as it may seem, three-dimensional space doesn't actually behave
like this! Mathematicians and physicists have long known that while a single
full turn entangles, two turns can be \textquotedblleft
untangled\textquotedblright. For physicists this provides an analogy for
understanding electron spin \cite[p. 28f]{Feynman}, \cite{FK,kauffman-video},
\cite[Fig. 41.6, p. 1149]{misner}, and for mathematicians it is a statement
about the topology of the rotation group $SO(3)$.

Visual demonstrations of untanglings
\cite{kauffman-video,wikipedia-orientation-entanglement,palais,wikipedia-plate-trick}
often have physical origins, with names like the \textquotedblleft Dirac
scissors demonstration\textquotedblright\ \cite[p. 43]{Penrose},
\textquotedblleft Dirac belt trick\textquotedblright, \textquotedblleft
Feynman plate trick\textquotedblright, or Indonesian or Philippine
\textquotedblleft candle dance\textquotedblright. While these fascinating
demonstrations convince one that untangling a double-twist is possible, they
have an ad hoc feel, and it is hard to see mathematically how they achieve an
untangling. For instance, if you are familiar with the belt trick (discussed
in detail in the next section), one thing you may notice is that however you
perform it, apparently some place on the belt undergoes a 180$%
{{}^\circ}%
$ rotation (ignoring any translation in 3-space that may occur) at some time
during the trick. But can one find a way to avoid this, or is it a
mathematical necessity? And are there other constraints that aren't easily
guessed just from watching such demonstrations? We will actually discover much
stronger constraints than this, and prove them all in our every-which-way
theorem of Section~\ref{EWW}.

Shortly we will explain that formally an untangling is a \emph{based
nullhomotopy} of the double-twist loop within the space of rotations of
three-space. Here are some questions we will then address: How efficient can a
based nullhomotopy be? For instance, what portion of the rotations in $SO(3)$
must an untangling utilize, and could each rotation utilized be used just
once, apart from those in the double-twist itself? How much must a based
nullhomotopy move individual vector directions around in three-space? And
finally, what does a maximally efficient based nullhomotopy \emph{look} like?
Our responses to these questions will include illustrative images and
supporting videos and animations. Even though at first you probably cannot
even imagine that it is possible, you should emerge able to \textquotedblleft
wave\textquotedblright\ a most efficient nullhomotopy of the double-twist with
your hand, and easily teach it to your friends.

Physical untangling demonstrations make the above questions difficult to
visualize, partly because the objects involved in the demonstration (Dirac's
belt, the dancer's arm) move in space during the demonstration, obscuring a
focus on the rotations involved. We might prefer to watch the untangling
process play out in a family of movies performed by an object, centered at the
origin, rotating about various axes through the origin. To achieve this, we
will first construct a particular based nullhomotopy directly from rotation
geometry.\footnote{Another based nullhomotopy is depicted in \cite{palais},
but is obtained by splicing three homotopies together, and lacks the
smoothness and efficiency we seek. An \emph{unbased} nullhomotopy is discussed
in \cite[pp. 788--9]{gray}, and two other based nullhomotopies in the
literature will be compared with ours a little later.} We will then utilize
properties of our nullhomotopy to deduce minimum constraints on the complexity
required by any based nullhomotopy, and will find that ours is as efficient as
possible subject to these constraints.

Our techniques involve degree theory, spherical geometry, and covering spaces,
and we emphasize the quaternionic viewpoint on rotations.

While reading, limber up your right hand and arm; we cannot be responsible for
sore muscles and joints. A second independent time coordinate would also be
very useful. If you have one, rev it up; if not, we will stretch your brain to
improvise one. Finally, if you just can't wait for visualization, animations,
and handwaving, then you can skip ahead, after Sections 2 and 3, to Sections 7
and 8.

\section{What does untangling involve?}

We need to make precise what we do and don't mean by \textquotedblleft
untangling.\textquotedblright\ The short answer is that untanglings will be
\emph{based nullhomotopies of loops} in the space of rotations of $3$-space,
but we will first explain the ideas more casually by describing and
illustrating two physical processes that give rise to untanglings.

\subsection{Physical untanglings}

We begin with the Dirac belt trick \cite{delatorre,egan}. Find a belt that is
not too stiff. Hold the buckle in your right hand and the other end in your
left hand, so that the belt runs straight between your hands. Now use your
right hand to twist the buckle through two full turns as you would a
right-handed screw, so that the belt twists about its length. It may appear
that without simply turning the buckle back the other way, you won't be able
to return the belt to its original, untwisted state.

However, it can be done. We will describe one method below, and you may find
your own way. The particular motions of your hands and arms are not important;
nor is the specific location of the belt in space. In fact, your arms are not
technically part of the challenge, so you are allowed to pass the belt
\textquotedblleft through\textquotedblright\ an arm. The key property of an
untangling is that once you're done twisting up the belt, you can perform the
untangling without changing the orientation at either end.

Here is one way to do it: Bring the ends together, pass the buckle around the
part of the belt hanging down from your left hand, and then pull the belt
straight again. During the maneuver, you will find that you momentarily need
to hold the buckle with your left hand, so that your right arm may move to the
other side of the belt (in effect, the belt will pass \textquotedblleft
through\textquotedblright\ your arm). If you pass the buckle in one direction
around the hanging belt, the tangling will increase, but if you pass it in the
other direction, the belt will come untangled. This process interpolates, in
time, between a double-twisted belt and an untwisted one. We demonstrate the
belt trick in the video \texttt{belt-trick.mp4} from the supplementary
materials accompanying this article (also at \cite{animations}).

Another untangling goes by the name of the waiter or Feynman plate trick, or
the Indonesian or Philippine candle dance
\cite{indonesian-candle-dance-youtube}. Extend your right hand straight ahead
of you, with your palm facing upwards (the waiter holds a full plate of food;
the dancer a lighted candle; a novice might choose a book). Now twist your arm
so that your palm, always staying horizontal, turns counterclockwise viewed
from above. After turning about $90^{\circ}$, you'll need to contort a bit to
continue, but by pushing your elbow out to the right you can comfortably swing
your hand beneath your armpit and continue to $270^{\circ}$, at which point
your hand will point out to the right and your arm will be twisted. You can
now bring your arm forwards so that your hand is back in its original
position, but your arm is awkwardly twisted with your elbow pointing upwards.
Maybe at this point you feel you have had enough, but we encourage you to
continue by twisting your arm so that your hand continues its counterclockwise
twisting, palm up. This time, your hand will pass over your head, and a rather
natural motion will return it to its original position, with your arm now
untwisted. We demonstrate the candle dance in the video
\texttt{candle-dance.mp4} from our supplementary materials (also at
\cite{animations}).

During the dance, your palm performs a double-twist while your shoulder
remains stationary, and your arm interpolates between these two motions. The
key property of the interpolation is that every point along your arm returns
to its original orientation at the end of its motion. In contrast to the belt
trick, here the interpolation between a double-twist and no twist is happening
not in time, but with respect to distance along your arm.

\subsection{Frames, loops, and homotopies}

Mathematically speaking, just what do these physical demonstrations tell us?
To tackle this, we will use the language of frames.

A (positive, orthonormal) \emph{frame} is a list of three mutually
perpendicular unit vectors in $\mathbb{R}^{3}$ which, when viewed as the
columns of a $3\times3$ matrix, lie in the group
\[
SO(3)=\{A:AA^{T}=I_{3}\text{ and }\text{det}(A)=1\}.
\]
In particular, the standard basis vectors $e_{1}$, $e_{2}$, and $e_{3}$ form a
frame, the \emph{standard reference frame}, whose matrix is the identity
matrix $I_{3}$.

Now begin again with your belt held straight from left to right, with both
width and length horizontal. Imagine a line drawn along the middle of the belt
from one end to the other, and attach a copy of the standard reference frame
at each point along this line, using the usual convention (right-hand rule)
that the positive $x$-, $y$-, and $z$-axes point (respectively) toward the
viewer, toward the right, and up. We view these frames as \emph{tangential
objects} consisting of vectors attached to the belt, so that any physical
positioning of the belt in $3$-space gives a path of frames, that is, a
continuous map of an interval into $SO(3)$. Notice that we care not about the
translational component of the positioning, but only about the rotational
positioning of each frame.

When you begin the trick with the belt twisted twice about its length, the
frames spin twice about their second vector as you progress along the belt. A
perfectly even double-twist yields the path of frames
\begin{equation}
\gamma(t)=\left[
\begin{array}
[c]{ccc}%
\cos(2t) & 0 & -\sin(2t)\\
0 & 1 & 0\\
\sin(2t) & 0 & \cos(2t)
\end{array}
\right]  , \label{rot}%
\end{equation}
where $t$ varies from $0$ at the left end to $2\pi$ at the buckle.

As you untangle the belt, at each time $s$ in the process the belt's position
gives another path $\gamma_{s}(t)$ inside $SO(3)$. The untangling process does
not change the frames at the ends of the belt: $\gamma_{s}(0)=\gamma_{s}%
(2\pi)=I_{3}$. Thus $\gamma_{s}$ is in fact a \emph{loop} in $SO(3)$, based at
$I_{3}$.

All in all, our untangling gives us a (continuous) \emph{based homotopy} of
loops in $SO(3)$:
\[
H\colon\,[0,1]\times\lbrack0,2\pi]\longrightarrow
SO(3),\,\,\,\,\,\,H(s,t)=\gamma_{s}(t).
\]
Since $\gamma_{0}$ is the double-twist around the $y$-axis and $\gamma_{1}$ is
the constant loop at $I_{3}$, $H$ is a \emph{based nullhomotopy} of the
double-twist. Notice that if we had simply allowed you to \emph{untwist} the
belt by turning the buckle back with an unscrewing motion, the paths
$\gamma_{s}$ would have $\gamma_{s}(0)=I_{3}$, but $\gamma_{s}(2\pi)\neq
I_{3}$.

We view the \emph{fundamental group} $\pi_{1}(SO(3))$ as the group formed by
based homotopy classes of loops of length $2\pi$, based at $I_{3}$. Recall
that the group operation is given by concatenating two loops and then
reparametrizing to form another loop of length $2\pi$. For example, the
single-twist $\tau_{y}$ around the $y$-axis is given by the loop
\[
\tau_{y}(t)=\left[
\begin{array}
[c]{ccc}%
\cos(t) & 0 & -\sin(t)\\
0 & 1 & 0\\
\sin(t) & 0 & \cos(t)
\end{array}
\right]  ,
\]
and $\tau_{y}^{2}$ is the loop in (\ref{rot}). The belt trick, then, provides
us with a physical proof that $\tau_{y}^{2}=1$ in $\pi_{1}(SO(3))$.

There is nothing special about our use of the $y$-axis as the rotation axis
for defining $\tau_{y}$. In fact, there is a continuous path in the unit
sphere $S^{2}$ connecting any two unit vectors, and the family of
single-twists around the vectors (viewed as rotation axes) in such a path
gives a based homotopy between the single-twists around the vectors at the two
ends of the path. Thus we are free, up to based homotopy, to twist about any
axis we like.

How does the candle dance translate into this mathematical formalism? With our
arm extended, ready to begin the dance, we attach the standard reference frame
at every spot along a line running along the arm from palm to shoulder to
neck. At every time during the dance, we obtain a path of frames along the
arm, or equivalently a path in $SO(3)$. But you will notice that these paths
are not loops: the frame at your neck stays fixed at the standard reference
frame, but the frame at your palm twists. On the other hand, fix a position on
your arm, and consider the path of frames that appear at that spot on your arm
as time varies. Since your arm starts and ends in the same position, this
gives a loop, based at the standard reference frame $I_{3}$! At your neck,
this loop is constant at $I_{3}$, while at your palm, this loop is the
double-twist about the vertical axis. Thus again we obtain a based
nullhomotopy of a double-twist, this time about the $z$-axis.

There is an important difference between our mathematical interpretations of
the belt trick and the candle dance. While the twisted belt is untangled in
time, the rotation of the candle is untangled along the arm. Thus in our
mathematical formalism for the belt trick, time represents the homotopy
parameter $s$, whereas for the candle dance, time functions as the loop
parameter $t$.

\subsection{Teasing apart a double-twist}%

\begin{figure}[tbh]
\centering\includegraphics[
width=3.3in
]
{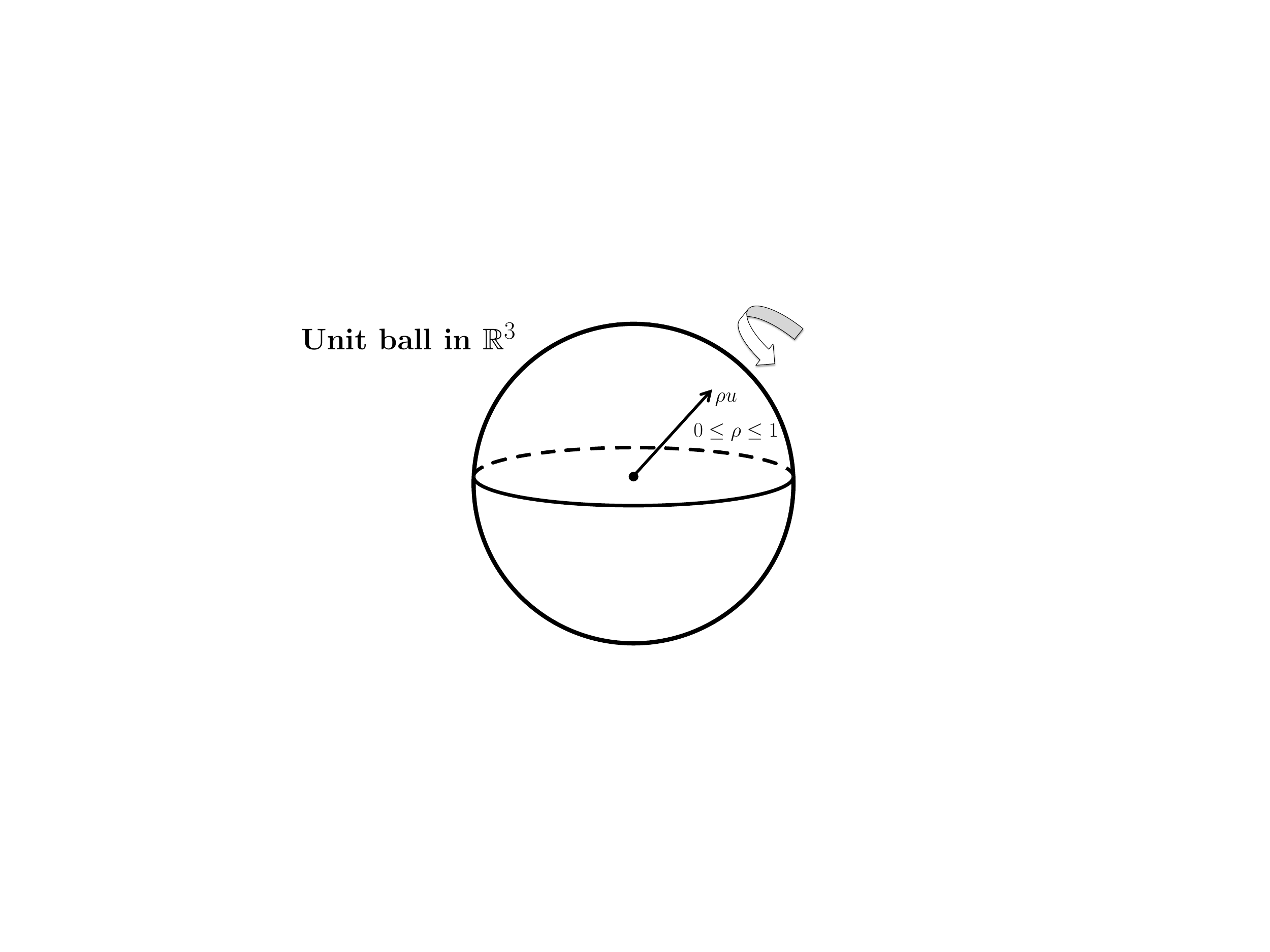}
\caption{Ball model in $\mathbb{R}^{3}$ for $SO(3)$}
\label{ball-model-description}
\end{figure}%

There is a nice model of $SO(3)$ created from the closed unit ball in
$\mathbb{R}^{3}$ (Figure \ref{ball-model-description}). A point in the ball
can be represented by $\rho u$, where $u\in S^{2}$ is a vector, and $\rho
\in\left[  0,1\right]  $ is a scalar. Let $\rho u$ correspond to rotation by
$\rho\pi$ radians about the axis determined by $u$, counterclockwise while
looking towards the origin from the tip of $u$. This correspondence provides a
homeomorphism between the ball and $SO(3)$, provided we identify antipodal
points on the boundary sphere, since rotation by $\pi$ is the same both
clockwise and counterclockwise\footnote{Implicit here and throughout is the
fact that for every matrix in $SO(3)$ the corresponding linear transformation
is a rotation of Euclidean space around some axis through the origin by some
angle. This was first proven by Leonhard Euler
\cite{euler,euler-wikipedia,euler-citizendium}, and can be deduced from the
spectral theorem in linear algebra \cite[Sec. 32]{curtis}.}. This quotient
space is also homeomorphic to real projective $3$-space $\mathbb{RP}^{3}$.%

\begin{figure}[tbh]%
\centering
\includegraphics[
natheight=7.499600in,
natwidth=9.999800in,
height=3.0277in,
width=4.0274in
]%
{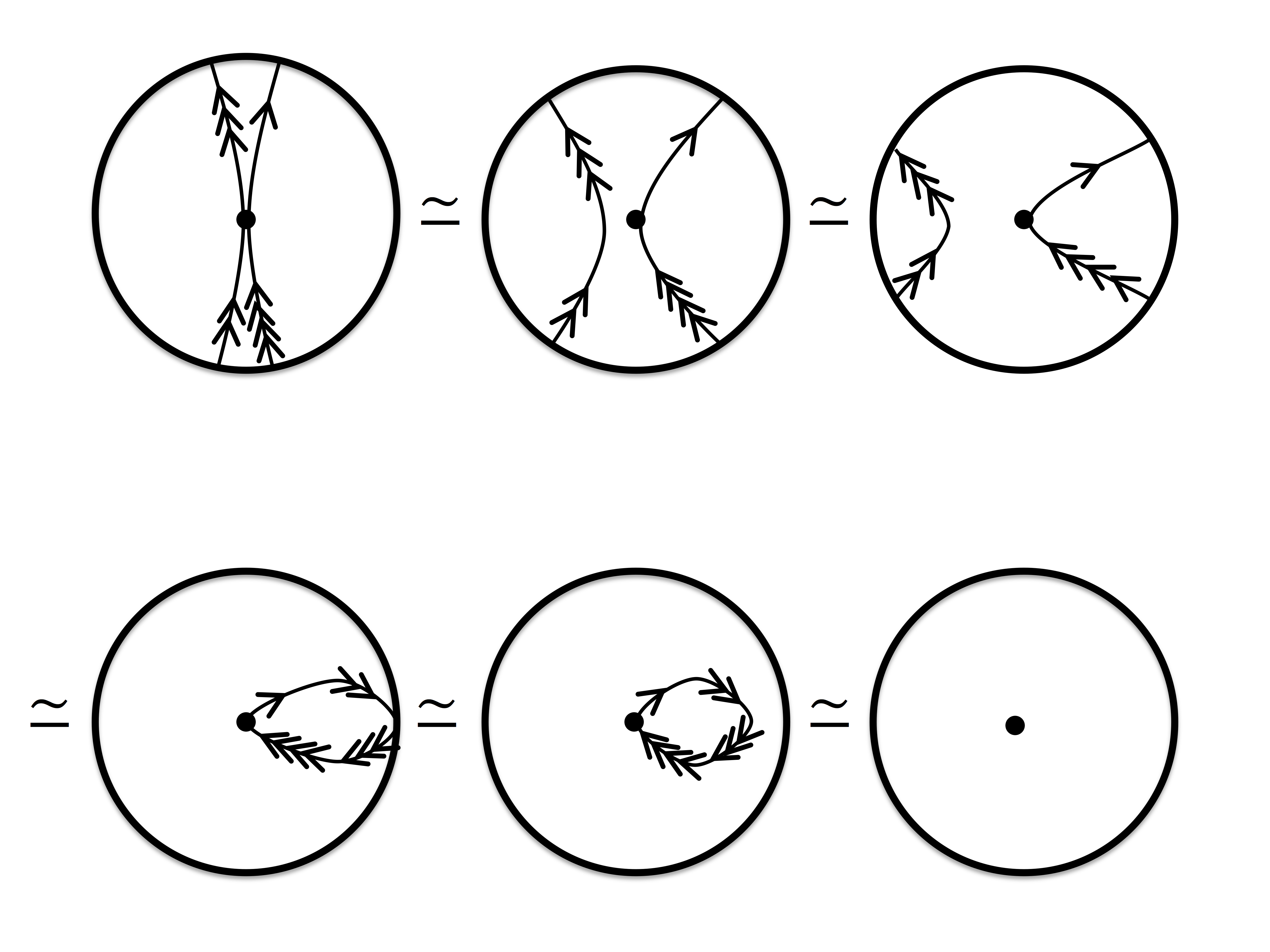}%
\caption{Nullhomotopy teasing apart double-twist diameters in a planar slice
of the ball model}%
\label{ball-model-nullhomotopy}%
\end{figure}

In the ball model, the double-twist around the axis $u\in S^{2}$ is a loop
based at the origin (which corresponds to the identity in $SO(3)$) that
traverses the diameter containing $u$ twice. Figure
\ref{ball-model-nullhomotopy} shows how to separate (tease apart) these two
traversals to create a based nullhomotopy.

Some of our original questions can be answered with the ball model. Clearly
the teasing-apart nullhomotopy lies inside the planar \textquotedblleft disk
model\textquotedblright\ subspace that utilizes only two of the three
coordinate directions for its rotation axes, and this subspace is homeomorphic
to $\mathbb{RP}^{2}$. Moreover, apart from the double-twist diameter itself,
each rotation in this subspace is utilized exactly once by the nullhomotopy.

However, despite these insights, the ball model is not sufficiently related to
rotation geometry to answer our other questions, which involve how the
rotations in a nullhomotopy of the double-twist move individual vectors around
in $3$-space.

\section{Conjuring a nullhomotopy from rotation geometry}%

\begin{figure}[tbh]%
\centering
\includegraphics[
natheight=2.863400in,
natwidth=9.422100in,
height=1.6025in,
width=5.2105in
]%
{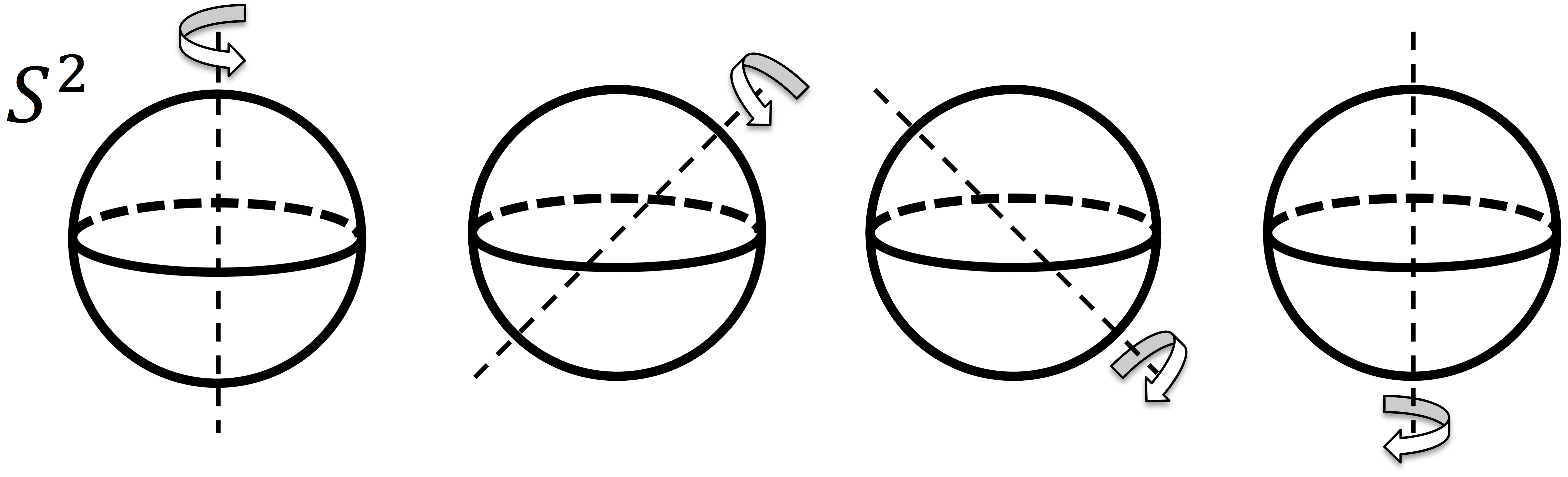}%
\caption{Tipping the axis upside down reverses twist}%
\label{single-twist-tipped}%
\end{figure}

How might we create a mathematically satisfying based nullhomotopy of the
double-twist out of simple rotation geometry? We will combine two ideas.

First, there is a relationship between counterclockwise (ccw) and clockwise
(cw) single-twists as based loops in $SO(3)$. Indeed, consider a ccw
single-twist $\tau_{z}$ around the positive $z$-axis. Now continuously tip the
axis of the twist along a great circle until the axis is upside down (Figure
\ref{single-twist-tipped}). This provides a based homotopy between a ccw
single-twist and a cw single-twist, both around the positive $z$-axis. From
this, the theory of fundamental groups will assure us that a double-twist is
nullhomotopic:
\[
\tau_{z}^{2}=\tau_{z}\tau_{z}\simeq\tau_{z}\tau_{z}^{-1}\simeq\text{constant
loop.}%
\]
While the latter homotopy is a standard construction for fundamental groups,
and pasting the two homotopies together provides an untangling, it is
inefficient and jerky.

A second idea is to replace concatenation in the double-twist by a different
point of view before deforming it. So far, we have been thinking of a
double-twist as the concatenation and reparametrization of two single-twists,
both around the same axis in the same rotation direction. Suppose instead that
for every parameter value $t$ we multiply\ using the group structure of
$SO(3)$, i.e., we compose the two rotations corresponding to $t$. The
resulting path will, at $t$, first rotate by $t$, and then rotate again by $t$
with the same axial ray and rotation direction, for a total rotation of $2t$.
Observe: The result of multiplying the two single-twists pointwise in $t$
(using the group structure of $SO(3)$) is exactly the same as concatenating
and reparametrizing them. Of course it is critical that the two single-twists
are around the same axis.%

\begin{figure}[tbh]%
\centering
\includegraphics[
natheight=2.883300in,
natwidth=8.395600in,
height=1.8386in,
width=5.2978in
]%
{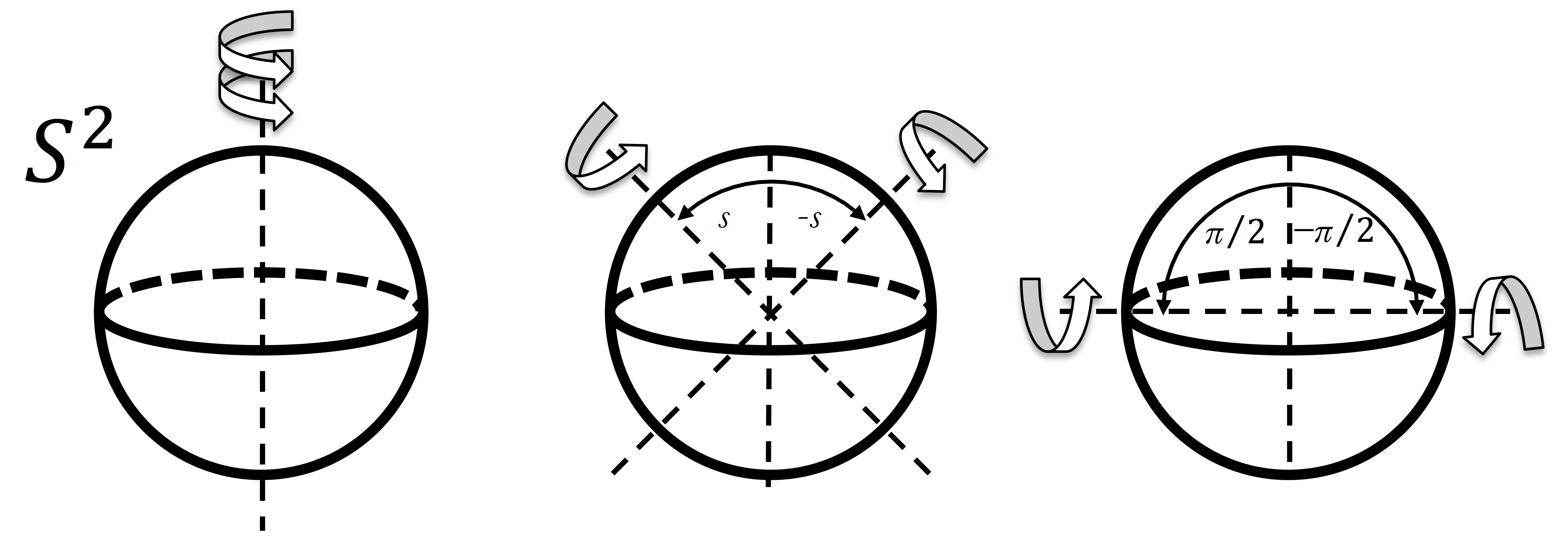}%
\caption{Double-tipping combined with multiplication in $SO(3)$}%
\label{double-tip}%
\end{figure}

We now combine these two ideas to create a homotopy $D$ (Figure
\ref{double-tip}). We begin at homotopy parameter value $s=0$ with the
counterclockwise double-twist around the positive $z$-axis, reinterpreted now
as the pointwise composition product of two ccw single-twists. As $s$
increases, we still use the composition product of two ccw single-twist loops
as our loop (recall that composition of rotations always produces another
rotation, per Euler), but we tip the axes of the two single-twist loops, and
we tip them differently. For a given $s$, we tip the first rotation axis
leftward by $s$ radians (in the plane $x=0$) from the positive $z$-axis toward
the negative $y$-axis, and we tip the second rotation axis rightward by $s$
radians toward the positive $y$-axis. Note that the result is always a based
loop: the paths being multiplied are always based loops, so at $t=0$ or $2\pi$
we are just composing the identity transformation with itself.

To see what our double-tipping homotopy $D$ is achieving, consider what
happens when $s$ reaches $\pi/2$, at which stage we have the composition
product of a ccw single-twist around the negative $y$-axis with a ccw
single-twist around the positive $y$-axis. For each $t$, we are composing two
inverse transformations, so this is the identity transformation for every $t$,
and thus we have the constant loop based at the identity. So we stop at
$s=\pi/2$ and declare victory.

Henceforth all nullhomotopies will be parametrized with $s$ ranging from $0$
to $\pi/2$, rather than to $1$, for easy comparison with our geometrically
defined nullhomotopy $D$.

\section{Quaternions describe the nullhomotopy}

The quaternions are astonishingly useful for describing rotations in 3-space
\cite[Ch. 8]{gallier}\cite{conway-smith,hanson,quaternions-spatial-rotation}.
We give just the bare essentials needed for our purposes.

As a vector space the \emph{quaternions} are just $\mathbb{R}^{4}$, with
elements $\left(  r,x,y,z\right)  $ written quaternionically as $r+xI+yJ+zK$.
The \emph{real part }is $r$, the \emph{imaginary part} is $xI+yJ+zK$, and the
\emph{conjugate} is $r-xI-yJ-zK$. The \emph{unit quaternions} are those with
Euclidean norm one, the \emph{pure quaternions} are those with zero real part,
and \emph{real quaternions }are those with zero imaginary part. The
quaternions have an associative, noncommutative multiplication in which the
real quaternions are central. This multiplication distributes over addition,
and is determined by the formulas%
\[
I^{2}=J^{2}=K^{2}=IJK=-1,
\]
from which one obtains
\[
IJ=K,\ \ JK=I,\ \ KI=J,\ \ JI=-IJ,\ \ KJ=-JK,\ \ IK=-KI.
\]
A quaternion times its conjugate yields the square of its norm, so quaternions
have multiplicative inverses, and the inverse of a unit quaternion is its
conjugate. The norm of a product is the product of the norms, so the unit
quaternions, which constitute the unit sphere $S^{3}$, are a group under multiplication.

The quaternions act linearly on themselves as follows: $qvq^{-1}$ is called
the conjugation action of $q$ on $v$. Since real quaternions commute with all
quaternions, scaling $q$ by a real quaternion leaves the conjugation
unchanged, so the action on $v$ is captured entirely by unit quaternions $q$,
and henceforth we consider only action by these. Even then, $q$ and its
antipodal point $-q$ still produce the same action on $v$. For unit
quaternions $q$, the conjugation formula $qvq^{-1}$ becomes $qv\overline{q}$,
where $\overline{q}$ denotes the conjugate of $q$.

An important, easily verified feature is that the pure quaternions are
preserved under the conjugation action. We shall focus on these, since they
are in clear correspondence with $\mathbb{R}^{3}$, with $I,J,K$ corresponding
to the standard unit vectors along the $x$--, $y$--, and $z$-- coordinate axes
(respectively). So each $\pm q$ in the unit quaternions $S^{3}$ produces, via
conjugation on the pure quaternions, a linear transformation of $\mathbb{R}%
^{3}$. What linear transformations are produced? If you are guessing that it
is the rotations, you are correct, and the details seem like magic.

Consider a unit quaternion $q$. Since it has norm one, it can be written
(almost) uniquely in the form $\cos\left(  \gamma/2\right)  +\sin\left(
\gamma/2\right)  u$, where $0\leq\gamma\leq2\pi$ and $u$ is a pure unit
quaternion, i.e., $u$ is a vector in $S^{2}$. The only caveat to uniqueness is
when $\sin\left(  \gamma/2\right)  =0$, which causes $u$ to be irrelevant.
This happens only when $\gamma=0$ or $2\pi$, i.e., when $q=\pm1$, in which
cases the conjugation action is the identity transformation. Now the following
amazing fact tells us exactly how conjugations produce rotations.

\begin{CRT}
Given a pure unit quaternion $u$, conjugation by
\[
\cos\left(  \gamma/2\right)  +\sin\left(  \gamma/2\right)  u
\]
is the linear transformation that rotates $\mathbb{R}^{3}$ around $u$, by
angle $\gamma$ counterclockwise (looking towards the origin from the tip of
$u$).
\end{CRT}

\noindent You may enjoy confirming this yourself.

\begin{unremark}
$\cos\left(  \gamma/2\right)  +\sin\left(  \gamma/2\right)  u$ is itself a
pure unit quaternion precisely when the angle of rotation under conjugation is
$\pi$ (that is, when $\gamma=\pi$).
\end{unremark}

We mentioned above that conjugation by $q$ and $-q$ produce the same
transformation, which can be effected in the correspondence of the theorem by
negating $u$ and changing the angle $\gamma$ to $2\pi-\gamma$, yielding the
same rotation described from the other end of the axis. Note too that the
quaternionic representation has $4\pi$--periodicity in $\gamma$, with
alteration by $2\pi$ effecting negation. It is clear from the formula that
negation is the only way that two different unit quaternions can produce the
same rotation. Thus conjugation creates a smooth two-to-one covering map from
the unit quaternions $S^{3}$ onto the space of rotations, i.e., onto $SO(3)$.
Let us denote this map sending $q$ to \textquotedblleft conjugation by
$q$\textquotedblright\ by $\mathcal{R}%
\co
S^{3}\rightarrow$ $SO(3)$.

One more superb feature of the double-covering $\mathcal{R}$ will enable us to
\textquotedblleft quaternionically compute\textquotedblright\ our
geometrically defined double-tipping nullhomotopy $D$. The homotopy was
defined using composition of rotations, i.e., the group structure for linear
transformations, corresponding to matrix multiplication in $SO(3)$. But recall
that the unit quaternions $S^{3}$ form a group, under quaternion
multiplication. Is the map $\mathcal{R}%
\co
S^{3}\rightarrow$ $SO(3)$ a group homomorphism? Yes, as seen from%
\begin{align}
\mathcal{R}\left(  q_{1}q_{2}\right)  (v)  &  =\left(  q_{1}q_{2}\right)
v\left(  q_{1}q_{2}\right)  ^{-1}=q_{1}q_{2}vq_{2}^{-1}q_{1}^{-1}\nonumber\\
&  =\mathcal{R}\left(  q_{1}\right)  \left(  \mathcal{R}\left(  q_{2}\right)
\left(  v\right)  \right)  =\left(  \mathcal{R}\left(  q_{1}\right)
\cdot\mathcal{R}\left(  q_{2}\right)  \right)  \left(  v\right)
\label{homomorphism}%
\end{align}

We need to compute the axes of rotation used in the double-tipping
nullhomotopy $D$. When we tip the unit vector $K$ on the positive $z$-axis by
$s$ radians toward the negative $y$-axis (whose unit vector is $-J$), this
yields the unit vector $-J\sin s+K\cos s$, and similarly yields $J\sin s+K\cos
s$ when we tip rightward toward the positive $y$-axis. Our definition applies
the leftward tipped rotation first, so according to the theorem and
(\ref{homomorphism}) we want the product
\begin{align*}
\widehat{D}\left(  s,t\right)  =  &  \left[  \cos\left(  t/2\right)
+\sin\left(  t/2\right)  \left(  J\sin s+K\cos s\right)  \right]  \cdot\\
&  \quad\quad\quad\quad\quad\quad\quad\quad\left[  \cos\left(  t/2\right)
+\sin\left(  t/2\right)  \left(  -J\sin s+K\cos s\right)  \right] \\
=  &  \left(  1-2\cos^{2}s\sin^{2}\left(  t/2\right)  \right)  +I\left(
\sin\left(  2s\right)  \sin^{2}\left(  t/2\right)  \right)  +K\left(  \cos
s\sin t\right)  ,
\end{align*}
a homotopy in the unit quaternions.

We may conclude that our double-tipping nullhomotopy $D\left(  s,t\right)  $
is conjugation by $\widehat{D}\left(  s,t\right)  $, i.e., $D\left(
s,t\right)  =\mathcal{R}\left(  \widehat{D}\left(  s,t\right)  \right)  $.

\begin{untheorem}
The mapping $D\left(  s,t\right)  $ into $SO(3),$ defined for $0\leq s\leq
\pi/2$ and $0\leq t\leq2\pi$ as conjugation on pure quaternions $xI+yJ+zK$ by
the unit quaternion
\begin{equation}
\widehat{D}\left(  s,t\right)  =\left(  1-2\cos^{2}s\sin^{2}\left(
t/2\right)  \right)  +I\left(  \sin\left(  2s\right)  \sin^{2}\left(
t/2\right)  \right)  +K\left(  \cos s\sin t\right)  ,
\label{quaternion-formula}%
\end{equation}
is a based nullhomotopy of the double-twist around $K$.\footnote{An equivalent
formula (after permuting coordinates) appears in \cite[pp. 121--2]{hanson},
but without indicating its geometric origins.}
\end{untheorem}

Before proceeding, we note that to compute $D\left(  s,t\right)  $ as a matrix
by carrying out the conjugation requires calculating 27 types of terms, each
involving up to triple products of $I,J,K$, then collecting and simplifying,
and produces entries in the 3 $\times$ 3 matrix that are sixth degree
polynomials in sines and cosines of double, single, and half angles of $s$ and
$t$!

\section{Just how \emph{does} our nullhomotopy untangle the double-twist?}

The most striking feature of our formula for $\widehat{D}$ is that it has no
$J$ component. So $\widehat{D}$ lands in the $r$-$x$-$z$ $2$-sphere, which we
call $S_{y=0}^{2}$; equivalently, $D$ uses only rotation axes lying in the
$x$-$z$ plane.\footnote{The nullhomotopy of the double-twist described in
\cite{FK,HFK}, despite the implication of the quaternion image picture in
\cite[Fig. 6]{HFK} that only two axial coordinate components are needed,
actually utilizes all three spatial coordinates for its axes, as seen in the
explicit coordinate formulas preceding Theorem 2 of \cite{FK}.
\par
Theirs, which we will call the FK-nullhomotopy after the authors, is, like
ours, derived by pointwise loop multiplication in $SO(3)$. Their nullhomotopy
multiplies the \textquotedblleft tip upside-down\textquotedblright\ homotopy
we described earlier, between ccw and cw single-twists, pointwise by a
single-twist in one chosen direction, which \textquotedblleft
shifts\textquotedblright\ the tipping homotopy to obtain a homotopy between a
double-twist and the identity. But the asymmetry of this shifting construction
produces a different result than our more symmetric double-tipping
construction.
\par
The precise relationship between the two nullhomotopies is as follows. Up to
permuting and negating some coordinates for the rotation axes, the
FK-nullhomotopy coordinate formulas can be obtained from ours by allowing our
$K$ axial component to spill into the $J$ direction, by rotating the $K$
component by $s$ radians toward $J$ during slice $s$ of the homotopy. So
instead of $K\left(  \cos s\sin t\right)  $ in the quaternionic homotopy
formula $\widehat{D}$, one would have the billowing rotation axis given by
$\left(  \sin s\right)  J\left(  \cos s\sin t\right)  +\left(  \cos s\right)
K\left(  \cos s\sin t\right)  $, along with an unchanged $I$ component.
\par
Thus, while the FK-nullhomotopy is based, the axes of each individual loop do
not limit at the ends of the loop to the vertical axis of the original twists,
whereas ours do (see below). This is one indication why the visualization and
waving of our homotopy will be easy to learn.} While this can be proven from
the double-tipping definition of $D$ via spherical geometry or linear algebra,
the quaternionic formula tells us immediately.

\begin{undefinition}
Let $P$ be the subspace of $SO(3)$ whose rotation axes lie in the $x$-$z$
plane. From the ball model we know that $P$ is diffeomorphic to the projective
plane $\mathbb{RP}^{2}$.
\end{undefinition}

%

\begin{figure}[tbh]%
\centering
\includegraphics[
natheight=7.499600in,
natwidth=9.999800in,
height=2.4275in,
width=3.2283in
]%
{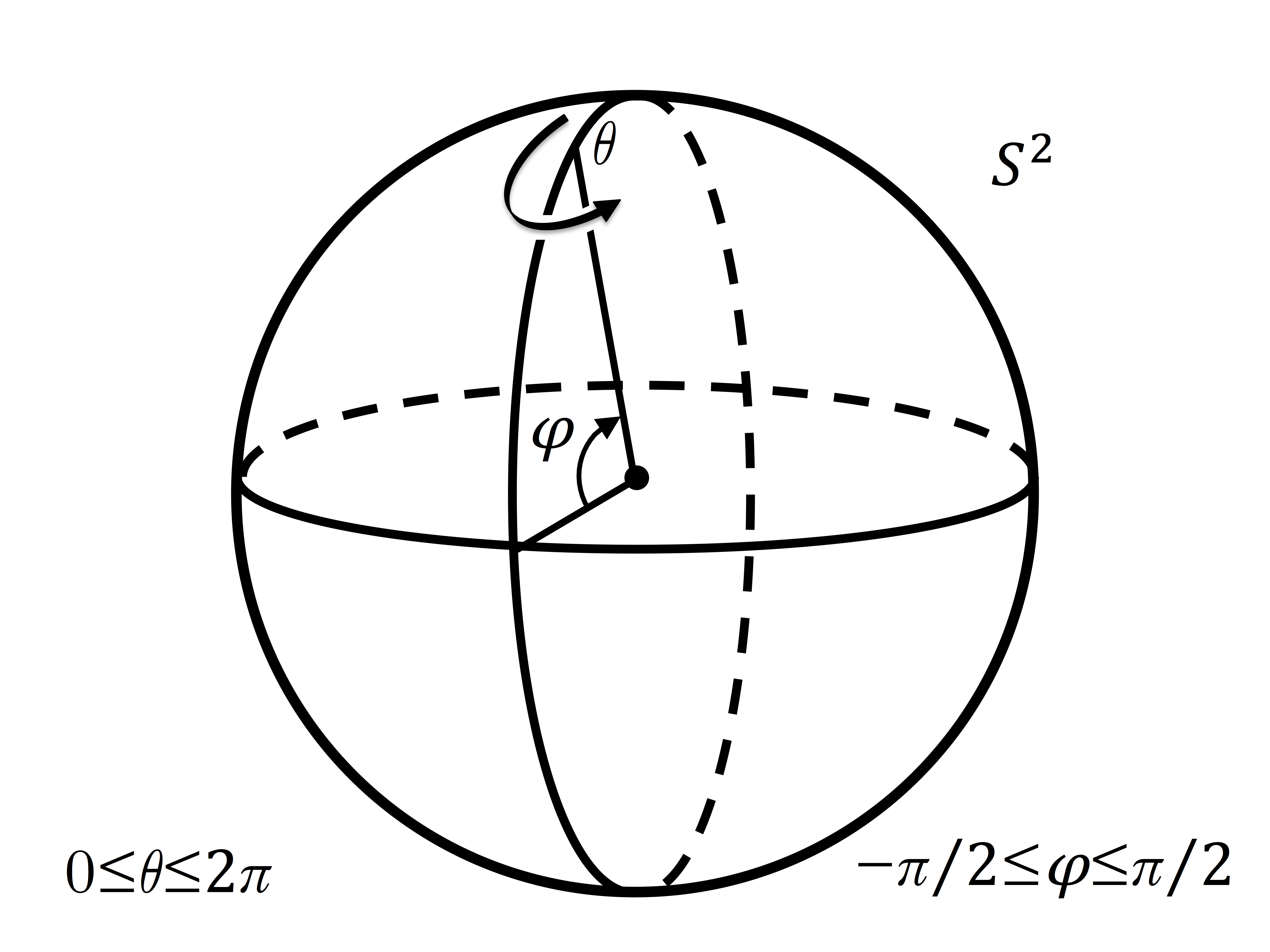}%
\caption{Coordinates $\varphi$ and $\theta$ for the subspace $P$}%
\label{phi-theta-defs}%
\end{figure}

We are interested in whether $D$, mapping into $P$, looks like the tease-apart
depiction in the ball model of Figure \ref{ball-model-nullhomotopy}. In
particular, does $D$ map onto $P$, and to what extent is it one-to-one, i.e.,
how efficiently does it use its rotations?

\subsection{Axial and rotational coordinates for the nullhomotopy}%

\begin{figure}[htb] \centering\begin{tabular}
[c]{cc}
{\parbox[b]{2.2546in}{\begin{center}
\includegraphics[
height=1.55in
]
{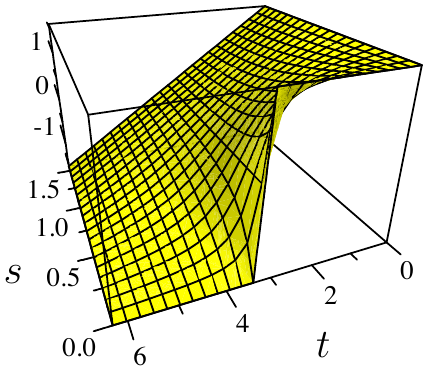}
\\
Axial angle $\varphi$ for $\protect\widehat{D}$
\end{center}}}
&
{\parbox[b]{2.3705in}{\begin{center}
\includegraphics[
height=1.55in
]
{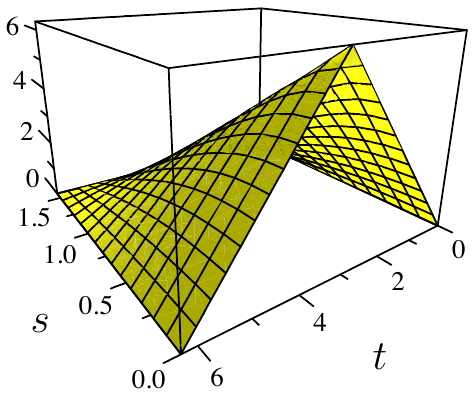}
\\
Rotation angle $\theta$ for $\protect\widehat{D}$
\end{center}}}
\end{tabular}
\caption{Axial and rotational angles for $\widehat{D}$}\label
{phi-theta-graphs}
\end{figure}%

We can get a sense for $D$ in terms of rotations by looking at graphs of its
rotation axes and angles, which serve as coordinates for $P$, aided by the
fact that $\widehat{D}$ uses only axial directions with nonnegative
$x$-coordinate (cf. the formula for $\widehat{D}$ in \ref{quaternion-formula}%
). Indeed, if we let $\varphi$ be the central angle in the $x$-$z$ plane
measured from the positive $x$-axis to the axis of rotation (Figure
\ref{phi-theta-defs}), then the formula for $\widehat{D}$ and the conjugation
rotation theorem yield
\[
\tan\varphi=\left(  \cos s\sin t\right)  /\left(  \sin\left(  2s\right)
\sin^{2}\left(  t/2\right)  \right)  =\csc s\cot\left(  t/2\right)  .
\]
And similarly, the ccw rotation angle $\theta$ around this axis is
\[
\theta=2\arccos\left(  1-2\cos^{2}s\sin^{2}\left(  t/2\right)  \right)  .
\]

Figure \ref{phi-theta-graphs} shows graphs for $\varphi$ and $\theta$ on the
domain rectangle of the homotopy $D$. Many features of $D$ can be understood
by studying these graphs. The reader may find it interesting to analyze why
one should conjecture therefrom that $D$ maps onto $P$, and also that $D$ is
one-to-one except for the requirements of a based nullhomotopy along the edges
of the domain rectangle. Additionally, one can consider how these graphs
achieve a based nullhomotopy of the double-twist. Any concern of a
discontinuity or lack of smoothness in the middle of the double-twist can be
assuaged by remembering that in $SO(3)$ the rotation axis is not well-defined
at the identity, and that close to the identity, all rotation axes appear with
small rotation amounts.

\subsection{The nullhomotopy captures a projective
plane\label{projective plane}}

The discussion immediately above suggests that the image of $D$ is all of $P$,
and indeed this is true for any based nullhomotopy landing in $P$.

\begin{untheorem}
Any based nullhomotopy of the double-twist landing in $P$ must map onto $P$.
\end{untheorem}

\begin{proof}
If not, then it would map into $P\approx\mathbb{RP}^{2}$, less a point, which
is homeomorphic to a M\"{o}bius strip, hence homotopy equivalent to a circle,
with the single-twist as generator of its fundamental group. But the
double-twist is not nullhomotopic in $P$ less the point, since $\pi_{1}\left(
S^{1}\right)  \cong\mathbb{Z}$.
\end{proof}

We have also conjectured from Figure \ref{phi-theta-graphs} that $D$ is
one-to-one, except for the identifications required of a based nullhomotopy
along the boundary of the rectangle. If true, we can then say that $D$ is
maximally efficient in its utilization of $SO\left(  3\right)  $.

\begin{untheorem}
$D$ is one-to-one except as prescribed along the edges of the homotopy rectangle.
\end{untheorem}

\begin{proof}
This can be proven from the geometric double-tipping definition of $D$ via
spherical geometry, but is more easily verified from the quaternionic formula
for $\widehat{D}$ on its rectangular domain $0\leq s\leq\pi/2$ and $0\leq
t\leq2\pi$, keeping in mind that $D=\mathcal{R}\circ\widehat{D}$, with
$\mathcal{R}$ identifying only negatives with each other, and noting that the
$I$ component of $\widehat{D}$ is never negative. The details are
straightforward using a few trigonometric identities, so we leave them to the reader.
\end{proof}

%

\begin{figure}[tbh]%
\centering
\includegraphics[
natheight=2.937800in,
natwidth=9.958300in,
height=1.2021in,
width=4.0119in
]%
{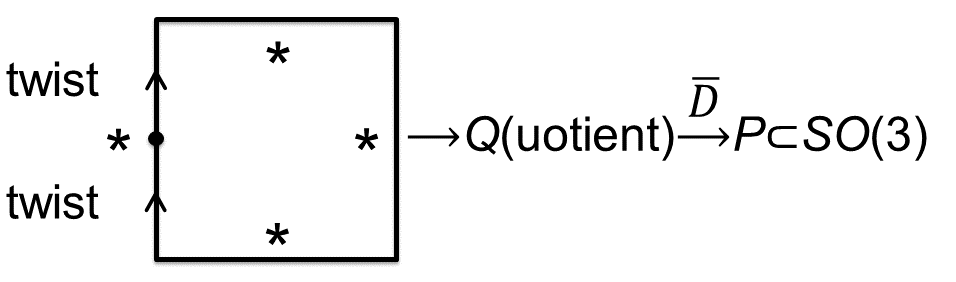}%
\caption{Any based nullhomotopy factors through $Q\approx\mathbb{RP}^{2}$}%
\label{rectangle-quotient}%
\end{figure}

We can now topologically combine the surjective and largely one-to-one
features of $D$. A based nullhomotopy of a double-twist must map one edge of
the rectangle twice around the single-twist in $SO(3)$, and the other three
edges to the identity transformation (Figure \ref{rectangle-quotient}). This
is equivalent to saying that the map factors through the quotient space $Q$ of
the rectangle formed by collapsing three edges to a point and wrapping the
remaining (now circular) edge twice around itself.

\begin{uncorollary}
$D$ induces a homeomorphism $\overline{D}%
\co
Q\rightarrow P\subset SO(3)$. Hence $Q$ is homeomorphic to $\mathbb{RP}^{2}$.
\end{uncorollary}

\subsection{Hemispherical models for the nullhomotopy}%

\begin{figure}[htb] \centering\begin{tabular}
[c]{cc}
{\parbox[b]{2.1369in}{\begin{center}
\includegraphics[
height=2.0928in
]
{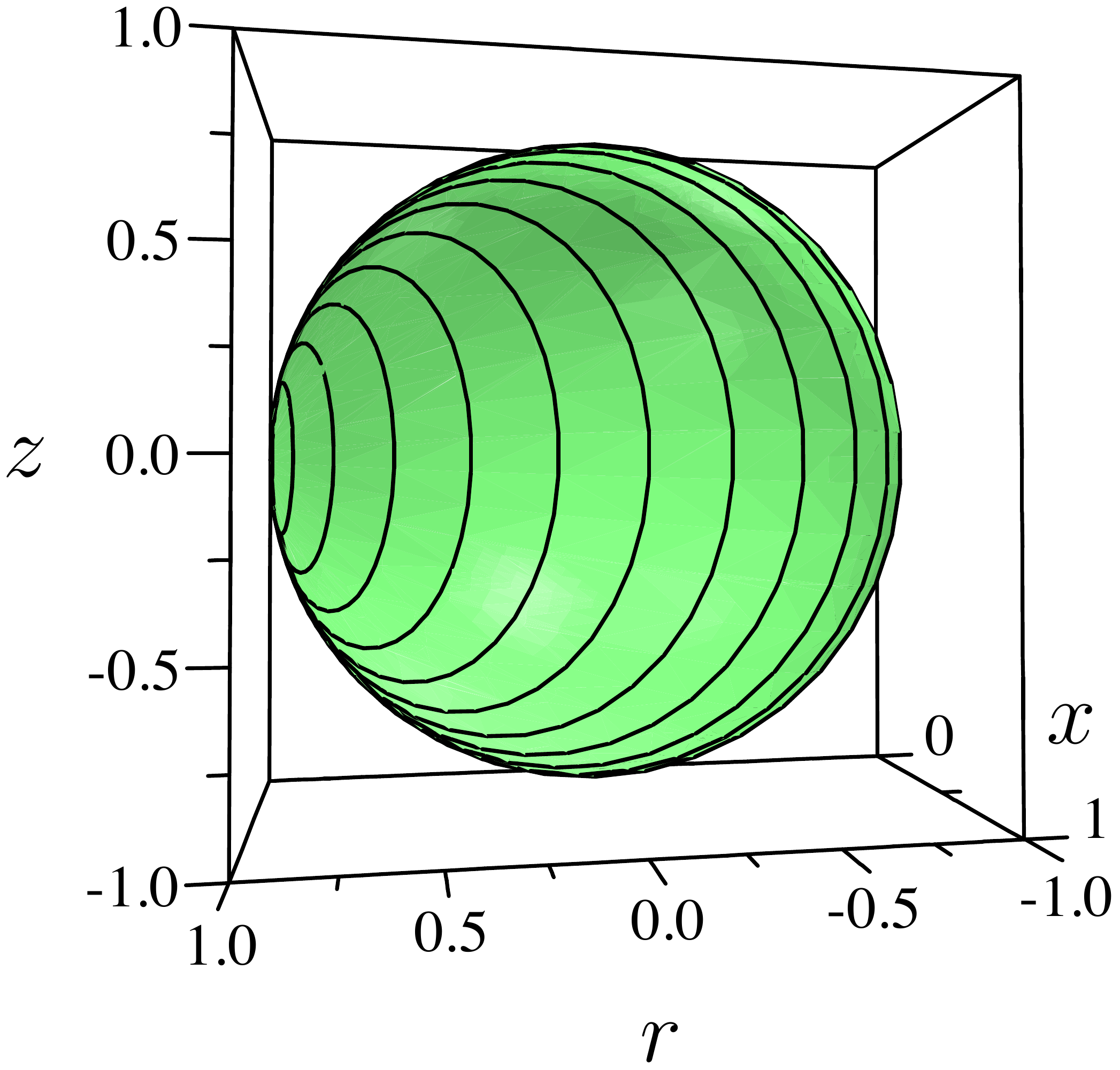}
\\
$D$ dragging loops across its lift to the $x$-nonnegative $2$-hemisphere
\end{center}}}
&
{\parbox[b]{2.5244in}{\begin{center}
\includegraphics[
height=2.0859in,
]
{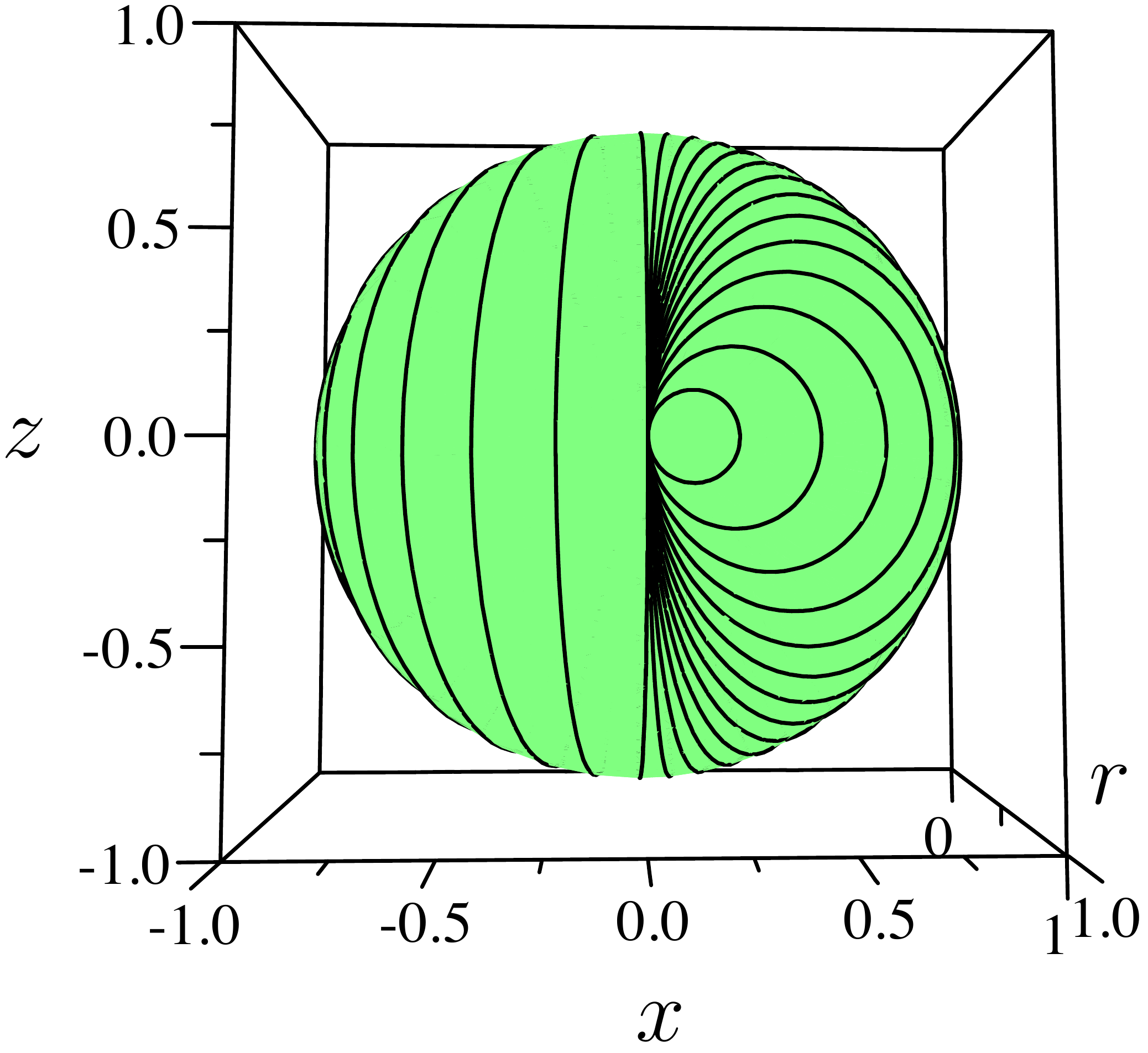}
\\
$D$ teasing apart the double-twist in its lift to the $r$-nonnegative
$2$-hemisphere
\end{center}}}
\end{tabular}
\caption{Two hemispherical views of $D$}\label{hemispheres}
\end{figure}%

It is illustrative to use hemispheres of the $S^{3}$ covering $SO(3)$ to view
from two very different perspectives how the lift $\widehat{D}$ behaves in the
$r$-$x$-$z$ 2-sphere $S_{y=0}^{2}$. Note that the image of $S_{y=0}^{2}$ under
the double-covering $\mathcal{R}$ is the subspace $P$ of $SO\left(  3\right)
$.

First, $\widehat{D}$ lands in the $x$-non-negative hemisphere of $S_{y=0}^{2}%
$, and the covering map $\mathcal{R}%
\co
S_{y=0}^{2}\rightarrow P$ makes the antipodal identification on the $r$-$z$
equator (left plot in Figure \ref{hemispheres}). The single-twist about the
$z$-axis lifts to $\cos\left(  t/2\right)  +\sin\left(  t/2\right)  K$ for
$0\leq t\leq2\pi$, traversing half a circle (so covering space theory shows
that the single-twist cannot be nullhomotopic). The double-twist lifts to
$\cos\left(  t\right)  +\sin\left(  t\right)  K$ for $0\leq t\leq2\pi$, and
$\widehat{D}$ is a nullhomotopy of this great circle loop. In Figure
\ref{hemispheres}, one sees $\widehat{D}$ drag this loop across the
$x$-non-negative hemisphere to the quaternion $1$.

For a second vantage point, and to compare with the tease-apart nullhomotopy
of Figure \ref{ball-model-nullhomotopy}, we can instead view $D$ lifted to the
$r$-non-negative hemisphere in $S_{y=0}^{2}$, with antipodal identification on
the $x$-$z$ equator (right plot in Figure \ref{hemispheres}). The right plot
is obtained from the left plot by rotating our view by $90^{%
{{}^\circ}%
}$, and then shifting the $r$-negative image points antipodally to make them
visible. In the right plot, the double-twist runs twice vertically through the
middle, as opposed to the left plot, where the double-twist runs around the
boundary. The right plot is satisfyingly similar to the teasing-apart
nullhomotopy in the disk model of Figure \ref{ball-model-nullhomotopy}.

It is interesting to contemplate how the extremely different plots in Figure
\ref{hemispheres} are just two hemispherical views of the same nullhomotopy in
$P$. A close study of the second plot can match it with the rotation geometry
information in the graphs of $\varphi$ and $\theta$ shown earlier.

\subsection{How does the nullhomotopy move vectors around?}

The theorems above tell us that, except for the double-twist itself, $D$
utilizes each rotation in $P$ exactly once. But how then do the rotations in
$P$ (those with axes in the $x$-$z$ plane) move individual vectors?

As an example, at the very center of the hemisphere depicted on the left in
Figure \ref{hemispheres}, and at the right edge of the hemisphere depicted on
the right, lies the quaternion $I$, so for just a single $\left(  s^{\prime
},t^{\prime}\right)  $ (actually the center of the domain rectangle) the
nullhomotopy $D$ conjugates frame vectors by $I$, resulting in rotation by
$\pi$ around $I$ (recall the remark after the conjugation rotation theorem).
So the standard frame is turned upside down, with $K$ sent to $-K$, $J$ sent
to $-J$, and $I$ fixed. Clearly there is no other way to send $K$ to $-K$
using an axis from the $x$-$z$ plane. So this actually proves that $K$ is sent
to $-K$ exactly once during the entire nullhomotopy $D$, and also illustrates
the reversal phenomenon described in the Introduction. On the other hand, we
see that $J$ is sent to $-J$ infinitely often by $D$, since rotation by $\pi$
around any axis in the $x$-$z$ plane does so.

\section{Where do all the vectors go?}%

\begin{figure}[htb] \centering\begin{tabular}
[c]{ll}
\raisebox{-0.0208in}{\parbox[b]{2.5598in}{\begin{center}
\includegraphics[
height=2.565in
]
{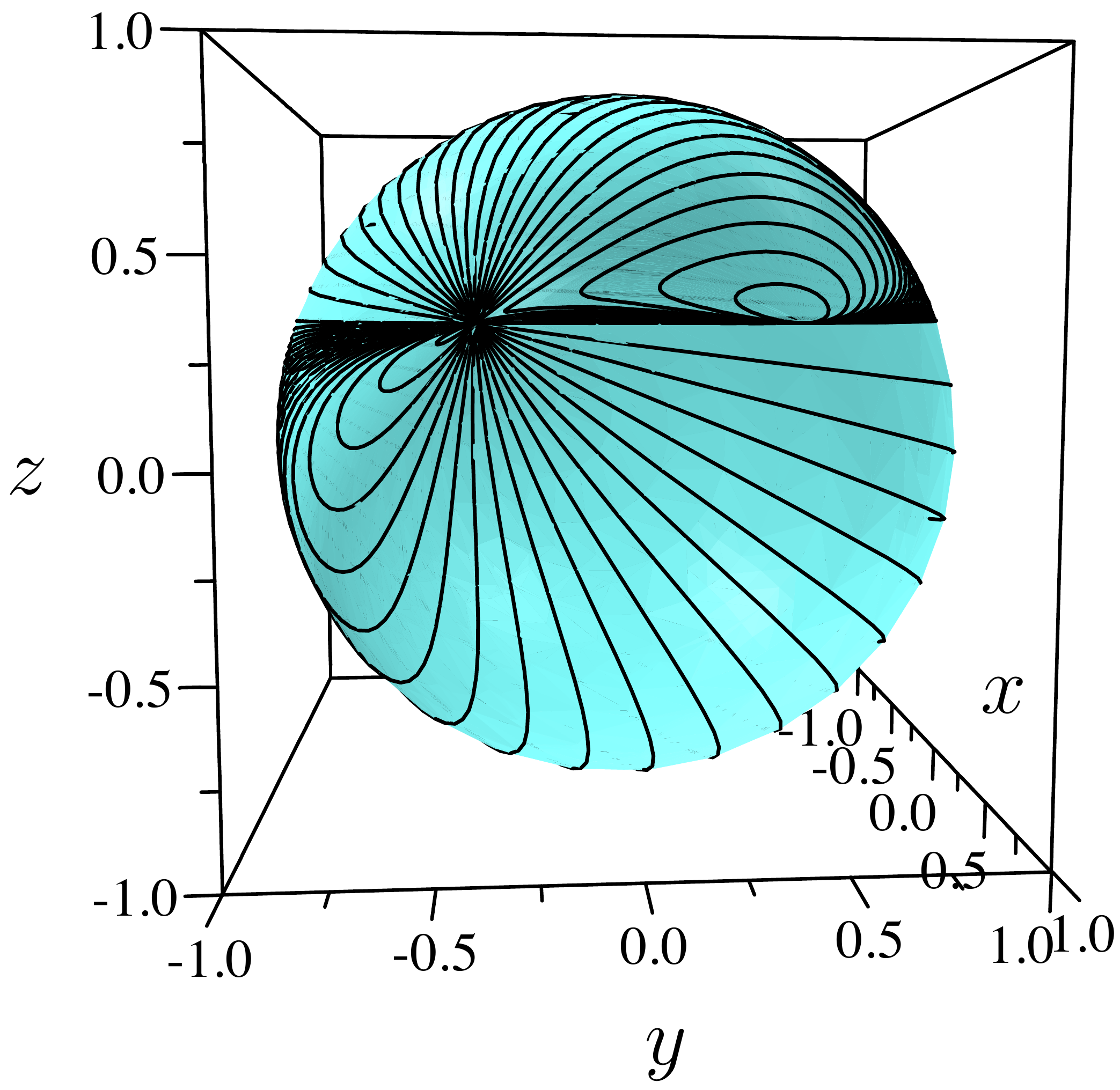}
\\
Path loops of $e_{v}\circ D$ in $S^{2}$, for $v=\left(  0.85,0.4,0.34278\right
)
$
\end{center}}}
&
{\parbox[b]{2.4091in}{\begin{center}
\includegraphics[
height=2.5516in
]
{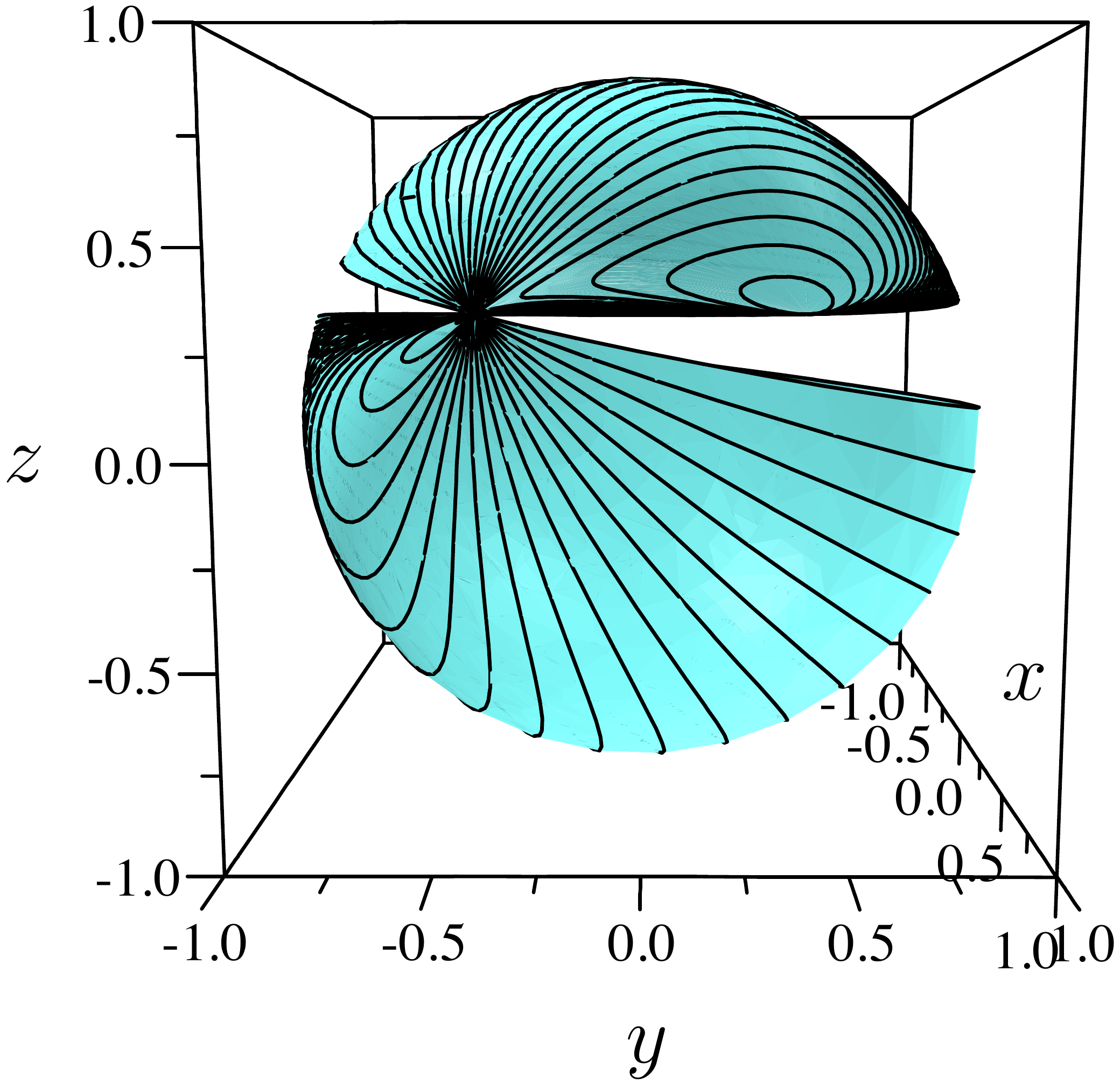}
\\
Path loops only for $s\geq0.1$: Clamshell with hinge
\end{center}}}
\end{tabular}
\caption{Evaluating $D$ on a unit vector  in $3$--space}\label{clam-shell}
\end{figure}%

In this section we consider the effect of $D$, and of other nullhomotopies of
the double-twist, on vectors in $S^{2}$. Although we have formulas for $D$,
and even have graphs of its rotation axes and angles, we don't have much sense
yet for how it carries frame vectors around.

We have seen that $D$ uses precisely all the rotations in the subspace $P$ of
$SO(3)$, and uses them most efficiently. We will analyze how rotations from
$P$ act as an ensemble on a vector, and how $D$ is related to other
nullhomotopies. Our conclusions will be visible in the animations we provide
for $D$, and will help guide us in learning how to wave $D$ with our hands.

\subsection{Collapsing $\mathbb{RP}^{2}$ via clamshells}

Given a unit vector $v\in S^{2}$, consider the evaluation map $e_{v}%
\co
SO(3)\rightarrow S^{2}$ defined by applying rotations in $SO(3)$ to $v$. We
are interested in the restriction $e_{v}|_{P}$ of $e_{v}$ to $P$. We can
visualize it by viewing the image in $S^{2}$ of the homotopy rectangle under
$e_{v}\circ D$, and displaying the successive loops in the composition
homotopy $e_{v}\circ D$. We pick a generic $v=\left(  0.85,0.4,0.34278\right)
$, and show the full image on the left side of Figure \ref{clam-shell}. All
loops in $S^{2}$ defined by $t\mapsto e_{v}\circ D\left(  s,t\right)  $ begin
and end at $v$, which lies on the right half of the prominent latitude. When
$s=0$, the double-twist loop carries $v$ twice around this latitude. On the
right side of Figure \ref{clam-shell} we restrict the homotopy parameter $s$
to begin at $.1$ rather than $0$. For $s=.1$, $v$ travels around the back of
the upper rim of the clamshell\footnote{Thanks to Pat Penfield for this
terminology.} to its \emph{hinge} $v^{\prime}$, then around the lower rim, and
finally back through the hinge to $v$, completing its traversal of the upper
rim. Intriguingly, the latitude of $v$ actually intersects the clamshell,
including its rims, only at $v$ and $v^{\prime}$. Most loops resemble $s=.1$,
traversing a \textquotedblleft figure eight\textquotedblright\ on $S^{2}$.
Only nearer the end of the homotopy, when $s$ is large, do the loops fail to
reach $v^{\prime}$, creating the loops in the upper shell near $v$.

Two striking observations stand out. First, the map $e_{v}|_{P}%
\co
P\rightarrow S^{2}$ appears to be surjective. Second, it appears to be mostly
one-to-one, except to the hinge point $v^{\prime}$. We will see that
$v^{\prime}$ is obtained from $v$ by negating its $y$-coordinate.

Of course no continuous map $P\rightarrow S^{2}$ could be surjective and
entirely one-to-one, since $P$ is homeomorphic to $\mathbb{R}\mathbb{P}^{2}$,
not to $S^{2}$. We now show that $e_{v}|_{P}$ collapses $P$ to $S^{2}$
efficiently, by putting all the non-injectivity on to the single vector
$v^{\prime}$.

\begin{untheorem}
The evaluation $e_{v}$ maps $P$ onto $S^{2}$, and is one-to-one on $P$ except
to the vector $v^{\prime}$ obtained from $v$ by negating its $y$-coordinate.
The inverse image of $v^{\prime}$ under $e_{v}%
\co
SO(3) \rightarrow S^{2}$ is an essential embedded circle\footnote{The map
$e_{v}\circ\mathcal{R}%
\co
S^{3}\rightarrow S^{2}$ is a Hopf map, and its circular fibers are the 2-fold
covers of the circular fibers of $e_{v}%
\co
SO(3) \rightarrow S^{2}$. Here we are seeing that exactly one of those fibers
lies in $P$, while the others each intersect it transversally in singletons.}
in $P$.
\end{untheorem}

\begin{proof}
First we prove the onto and one-to-one claims. For an arbitrary $w\in S^{2}$,
consider the rotations carrying $v$ to $w$. If $w=v$, then the identity, which
is in $P$, maps $v$ to $w$, and the only other rotations to do so are
nonidentity rotations around $v$. These are only in $P$ if $v$ lies on the
$x$-$z$ equator, in which case $w=v=v^{\prime}$, and there is no claim about
injectivity. If $w\neq v$, then the axes of the rotations carrying $v$ to $w$
lie on the perpendicular bisector of the great arc(s) joining $v$ and $w$. If
this perpendicular bisector (a great circle) is different from the $x$-$z$
equator, then it intersects the $x$-$z$ equator in two antipodal points,
yielding exactly one rotation in $P$ carrying $v$ to $w$. If on the other hand
this perpendicular bisector is the $x$-$z$ equator, then $w$, being the
reflection of $v$ across the perpendicular bisector, is $v^{\prime}$, and
there is again no claim about injectivity.

Now we consider the inverse image of $v^{\prime}$ under $e_{v}$. If
$v^{\prime}=v$, then $v$ is on the $x$-$z$ equator, and the inverse image of
$v^{\prime}$ consists of the rotations around $v$ itself, a circle in $P$. If
on the other hand $v^{\prime}\neq v$, then we consider the great arc $A$
joining $v$ to $v^{\prime}$, and by reflecting in $y$ we see that $A$ has the
$x$-$z$ equator $E$ as perpendicular bisector. The axes of the rotations
carrying $v$ to $v^{\prime}$ all lie on $E$. Where $A$ intersects $E$, a
rotation by $\pi$ will carry $v$ to $v^{\prime}$, and for every other axis
along $E$ there is one rotation carrying $v$ to $v^{\prime}$, with rotation
angle varying continuously with the axis. In the one case where $v$ and
$v^{\prime}$ are antipodal ($v=J$) the rotation angle remains as $\pi$
throughout. Otherwise the rotation angle varies monotonically to an extreme
when the rotation axis is perpendicular to the plane containing $A$, then
returns to $\pi$ as the axis completes one full flip. In all cases the inverse
image is homeomorphic to a circle in $P$, and it corresponds to a
homotopically essential loop, since its lift to the universal cover $S^{2}$
joins antipodal points\footnote{A similar geometric argument shows that for
all $w$ in $S^{2}$, the set $C(v,w)$ of transformations in $SO(3)$ carrying
$v$ to $w$ is an essential circle. This can also be seen group-theoretically,
since this set of transformations is a coset of the subgroup $C$ fixing $v$,
and $C$ is just the loop of rotations with axis $v$. From this perspective, a
computation in the quaternions shows that the loop $C(v,w)$ is transverse to
$P$ when $w\neq v^{\prime}$.}.
\end{proof}

\subsection{Vectors must fly every-which-way}

$\label{EWW}$

Recall (Section \ref{projective plane}) that $D$ factors through a
homeomorphism $\overline{D}%
\co
Q\rightarrow P$. Since $P$ is diffeomorphic to $\mathbb{RP}^{2}$, we may
consider $Q$ with the differentiable structure making $\overline{D}$ a
diffeomorphism. Of course $e_{v}$ is also smooth.

\begin{uncorollary}
Both $e_{v}|_{P}$ and $e_{v}\circ\overline{D}$ have nonzero mod two degree as
smooth maps between differentiable $2$-manifolds.
\end{uncorollary}

\begin{proof}
Mod two degree may be computed as the mod two cardinality of the inverse image
of any regular value \cite{hirsch}. Regular values are dense (Sard's Theorem),
so we may choose a regular value $w\neq v^{\prime}$. From the theorem above,
the inverse image of $w$ is a singleton.
\end{proof}

The next result relates arbitrary based nullhomotopies to $D$.

\begin{untheorem}
Any two based nullhomotopies of the double-twist in $SO(3)$ are homotopic
through based nullhomotopies of the double-twist.
\end{untheorem}

\begin{proof}
Construct a map from the boundary of a box to $SO(3)$, having the two given
nullhomotopies on opposite faces, and the boundary requirements of based
nullhomotopies of the double-twist on the edges of the slices of the box
parallel to these faces. Since the boundary of the box is homeomorphic to
$S^{2}$, this map from the boundary to $SO(3)$ can be filled in continuously
on the interior of the box provided the second homotopy group of $SO(3)$ is
trivial. It is, since $S^{3}$ is a double cover of $SO(3)$, and $S^{3}$ has
trivial second homotopy group. Now the maps on the slices of the filled-in
box, parallel to the two opposite faces, provide the desired homotopy between
the two original nullhomotopies.
\end{proof}

We now prove that every based nullhomotopy of the double-twist must send every
vector in every direction to achieve the untangling.

\begin{EWW}
Let $H$ be a based nullhomotopy of the double-twist in $SO\left(  3\right)  $.
For all vectors $v,w\in S^{2}$, there exist $s$ and $t$ such that
$H(s,t)(v)=w$.
\end{EWW}

\begin{proof}
The slices of a homotopy between $H$ and $D$, as given by the theorem above,
create a homotopy between $\overline{H}$ and $\overline{D}$ as maps from $Q$
to $SO(3)$. Then $e_{v}\circ\overline{H}%
\co
Q\rightarrow S^{2}$ is homotopic to $e_{v}\circ\overline{D}$, so from the
corollary above, $e_{v}\circ\overline{H}$ has nonzero mod two degree as a
continuous map between differentiable $2$-manifolds (because
$\operatorname{mod}2$ degree is a homotopy invariant \cite{hirsch}). Thus
$e_{v}\circ\overline{H}$ must be onto, as non-surjective maps have degree
zero. Since $e_{v}\circ\overline{H}(s,t)=H(s,t)(v)$, this completes the proof.
\end{proof}

\begin{unremark}
With a bit more work, one can actually show that this theorem holds true for
arbitrary nullhomotopies of the double-twist, based or not. The key point is
that the boundary requirements of an unbased nullhomotopy still provide a
factorization through $\mathbb{RP}^{2}$ (using a somewhat different model from
$Q$). We leave the details as an exercise for the interested reader.
\end{unremark}

\section{Visualizing the double-tipping nullhomotopy}

We can visualize a based loop in $SO\left(  3\right)  $ as a movie in which an
object embedded in space rotates, returning to its original embedding. A based
nullhomotopy of the double-twist is then a continuous family of movies
beginning with a movie in which an object rotates twice about a fixed axis,
and ending with a movie in which the object never moves. We know such a family
of movies is impossible starting with a single rotation of an object. But for
a double rotation, our nullhomotopy $D$ will provide such a family of movies.
Of course there are infinitely many movies in this family, one for each value
of the homotopy parameter $s$. However, we will find that it is sufficient to
view these movies for just $8$ increments of $s$.

There are three ways the reader may view movies of $D$:

\begin{itemize}
\item Movies for the $8$ increments are displayed in synchrony in the video
\texttt{doubletip-nullhomotopy.mp4} from our supplementary materials (also at
\cite{animations}).

\item The program in the folder \texttt{Rotations2.0} in our supplementary
materials (also at \cite{animations,MEGLab}) was created by students in the
George Mason University Experimental Geometry Lab (MEGL) \cite{MEGLab}. We are
grateful to them for this wonderful creation. It provides the user with
complete control of the parameters $s$ and $t$ of a movie of $D$ in real time.
This allows the viewer to explore many features of $D$ related to our
discussions here. Begin with their \texttt{README2.0} file, which explains how
to control not only these parameters, but also other options discussed below.

\item Here in this article we display a grid of frozen images from the movies.
\end{itemize}

\renewcommand{\arraystretch}{0}\setlength{\tabcolsep}{.03in}%

\begin{table}[htbp] \centering
\begin{tabular}
[c]{ccccc}%
{\includegraphics[
natheight=4.166700in,
natwidth=4.147000in,
height=0.9447in,
width=0.9414in
]%
{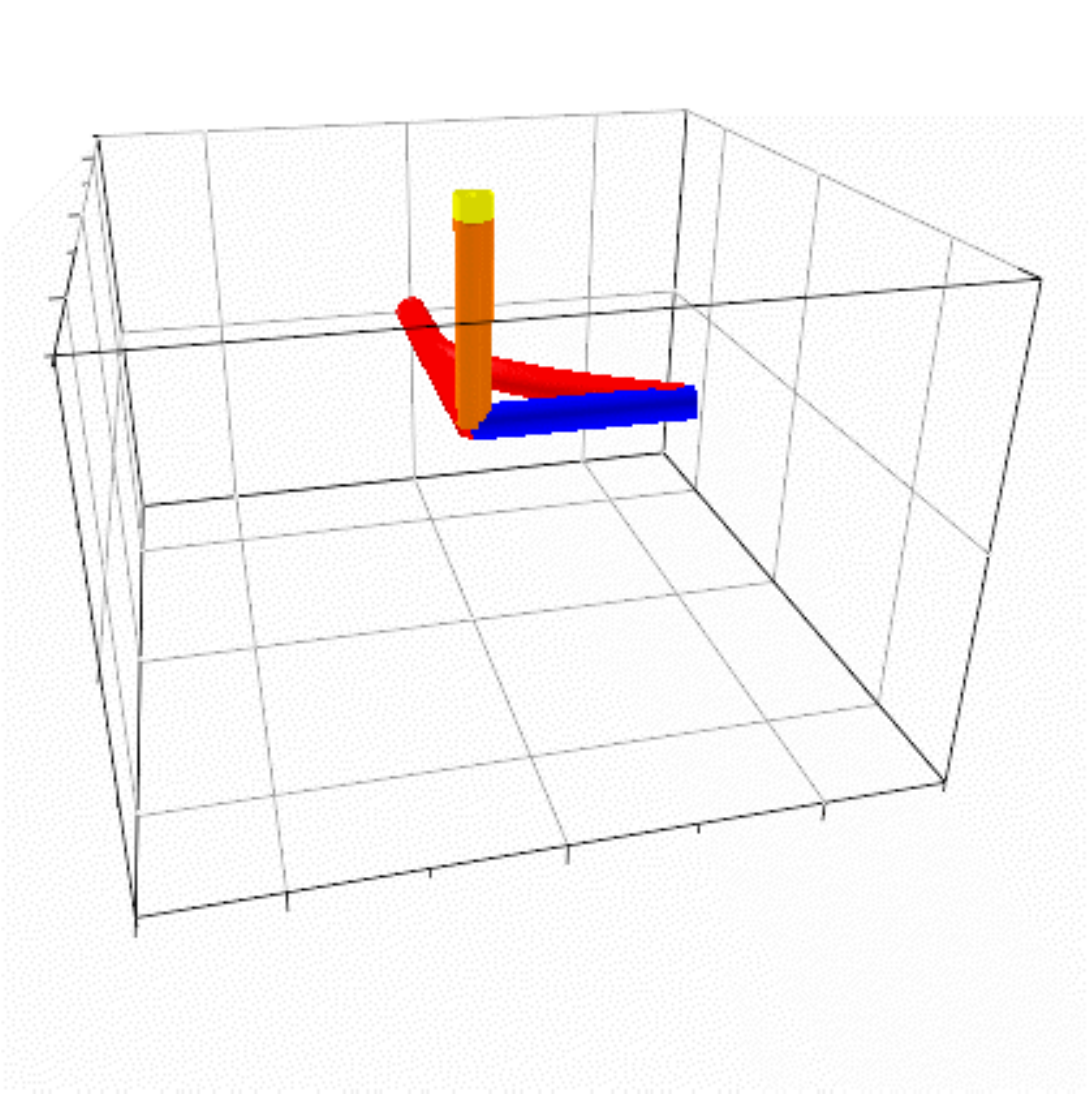}%
}
&
{\includegraphics[
natheight=4.166700in,
natwidth=4.147000in,
height=0.9447in,
width=0.9414in
]%
{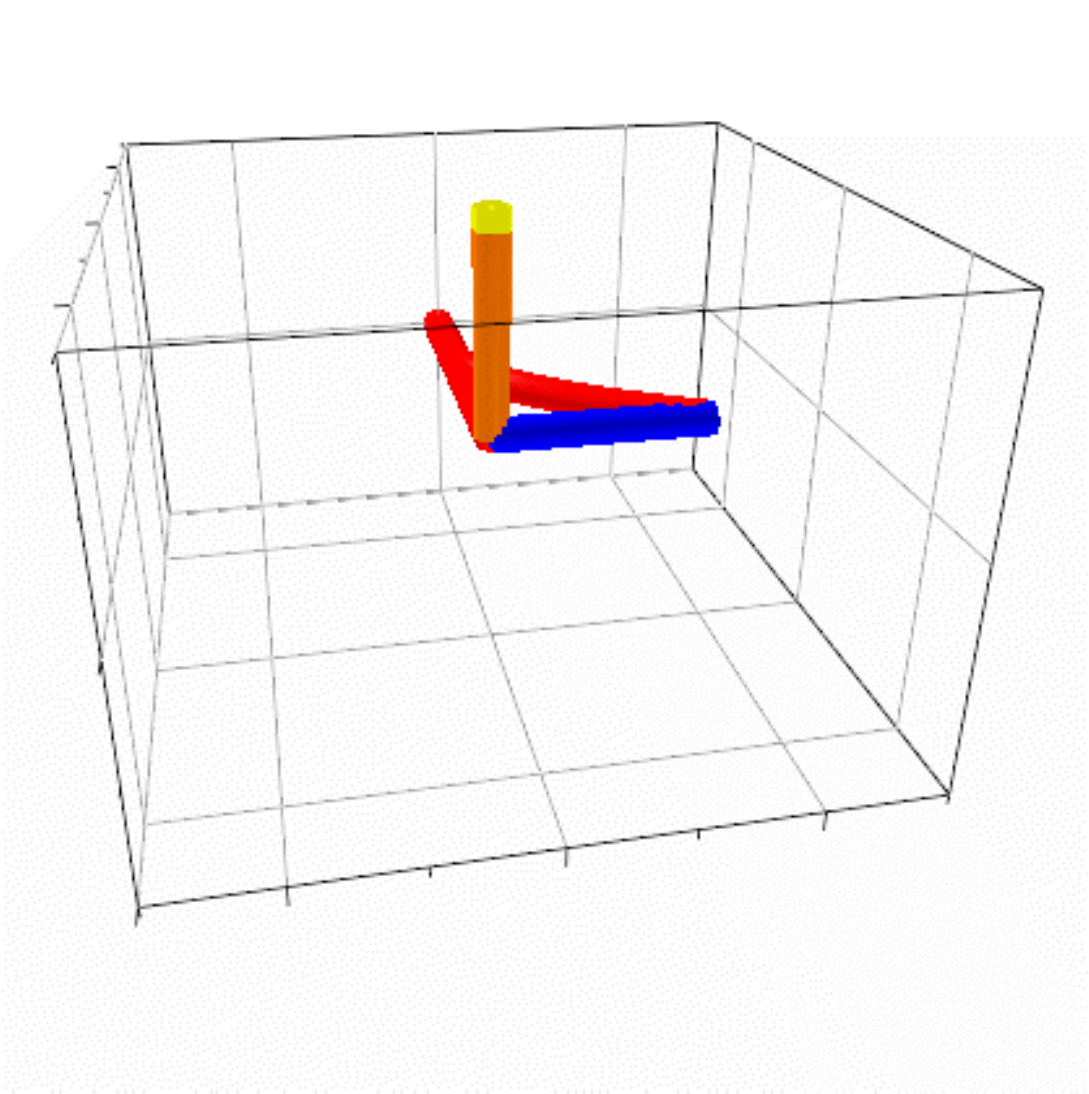}%
}
&
{\includegraphics[
natheight=4.166700in,
natwidth=4.147000in,
height=0.9447in,
width=0.9414in
]%
{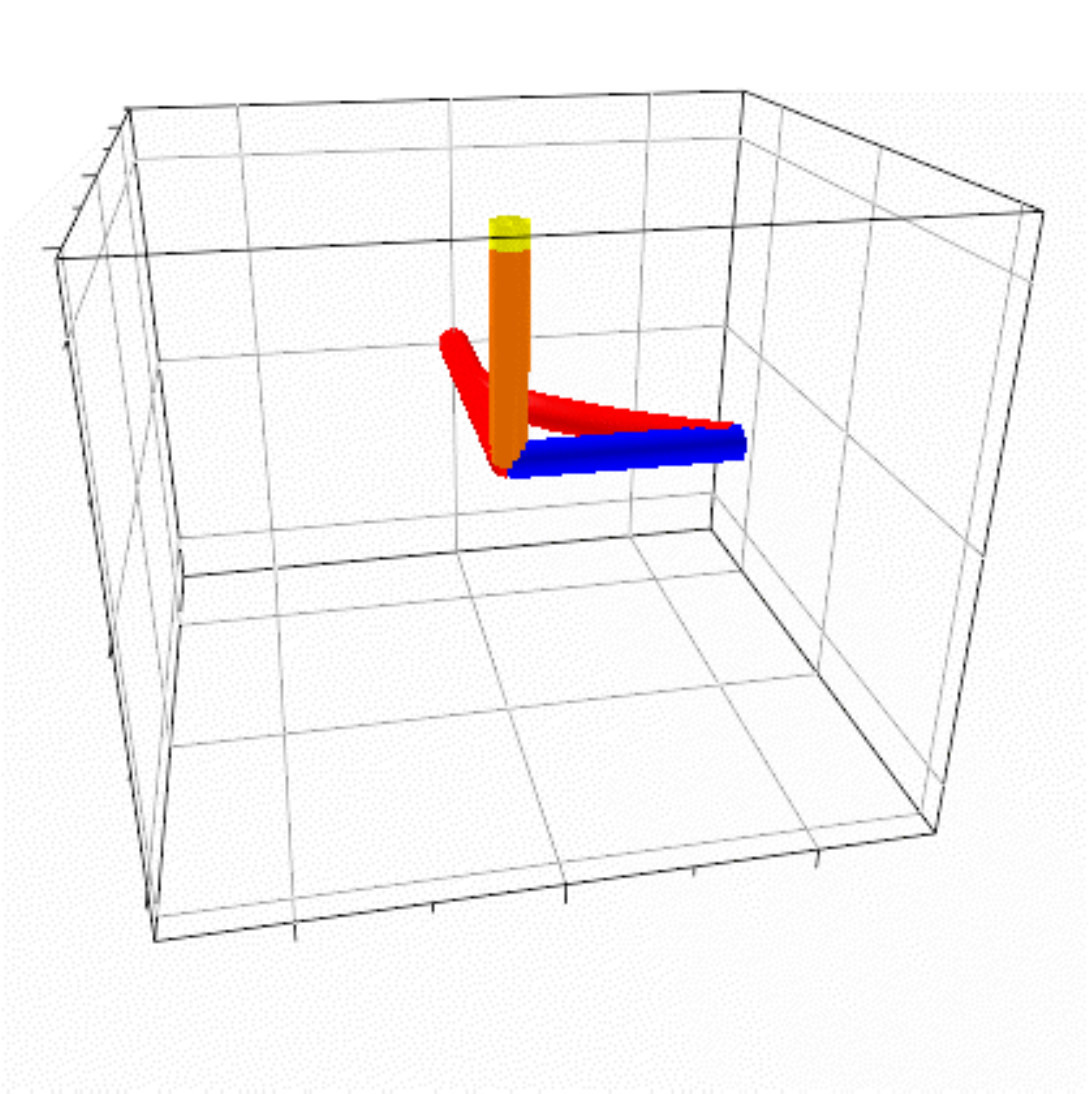}%
}
&
{\includegraphics[
natheight=4.166700in,
natwidth=4.147000in,
height=0.9447in,
width=0.9414in
]%
{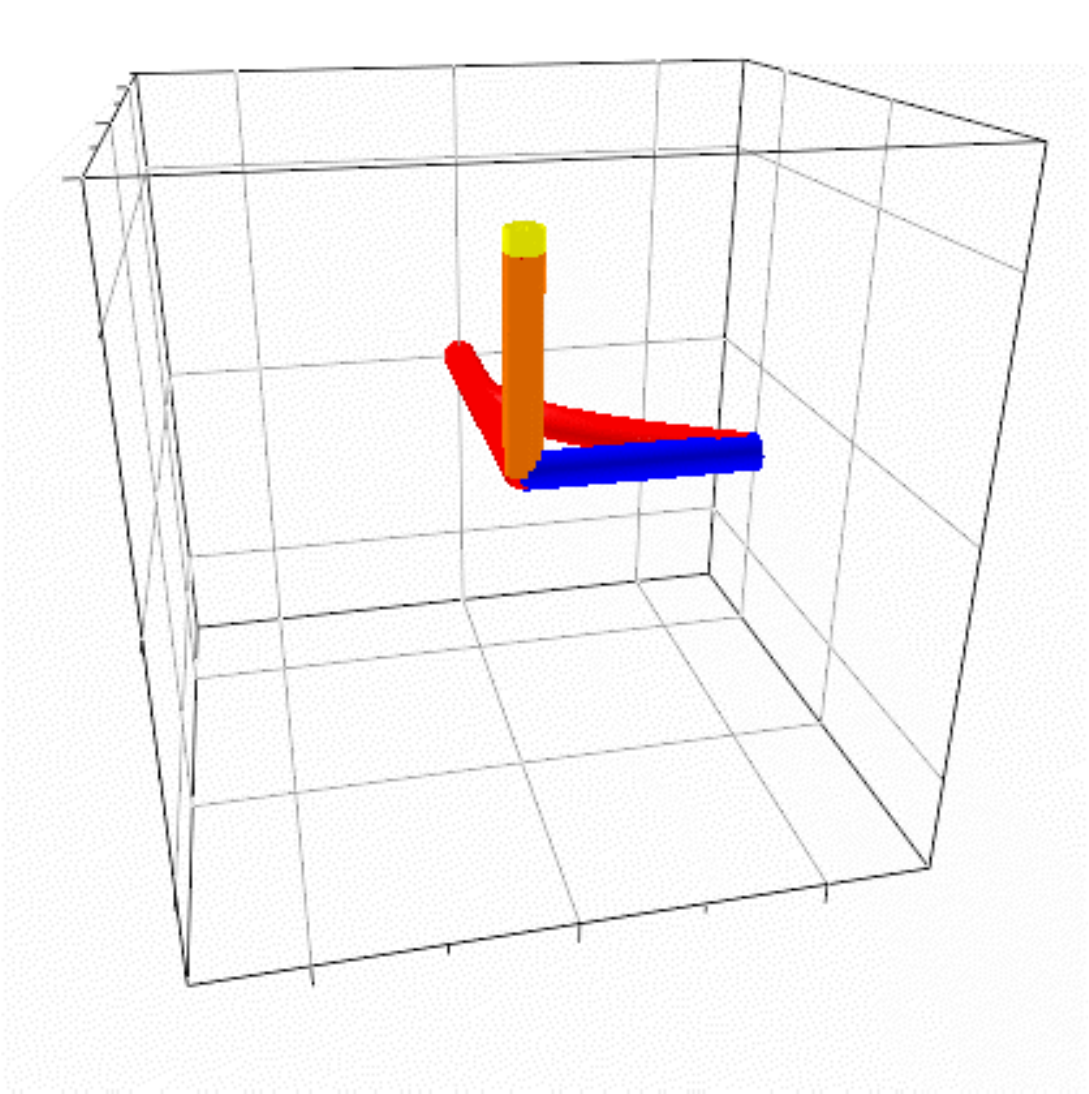}%
}
&
{\includegraphics[
natheight=4.166700in,
natwidth=4.147000in,
height=0.9447in,
width=0.9414in
]%
{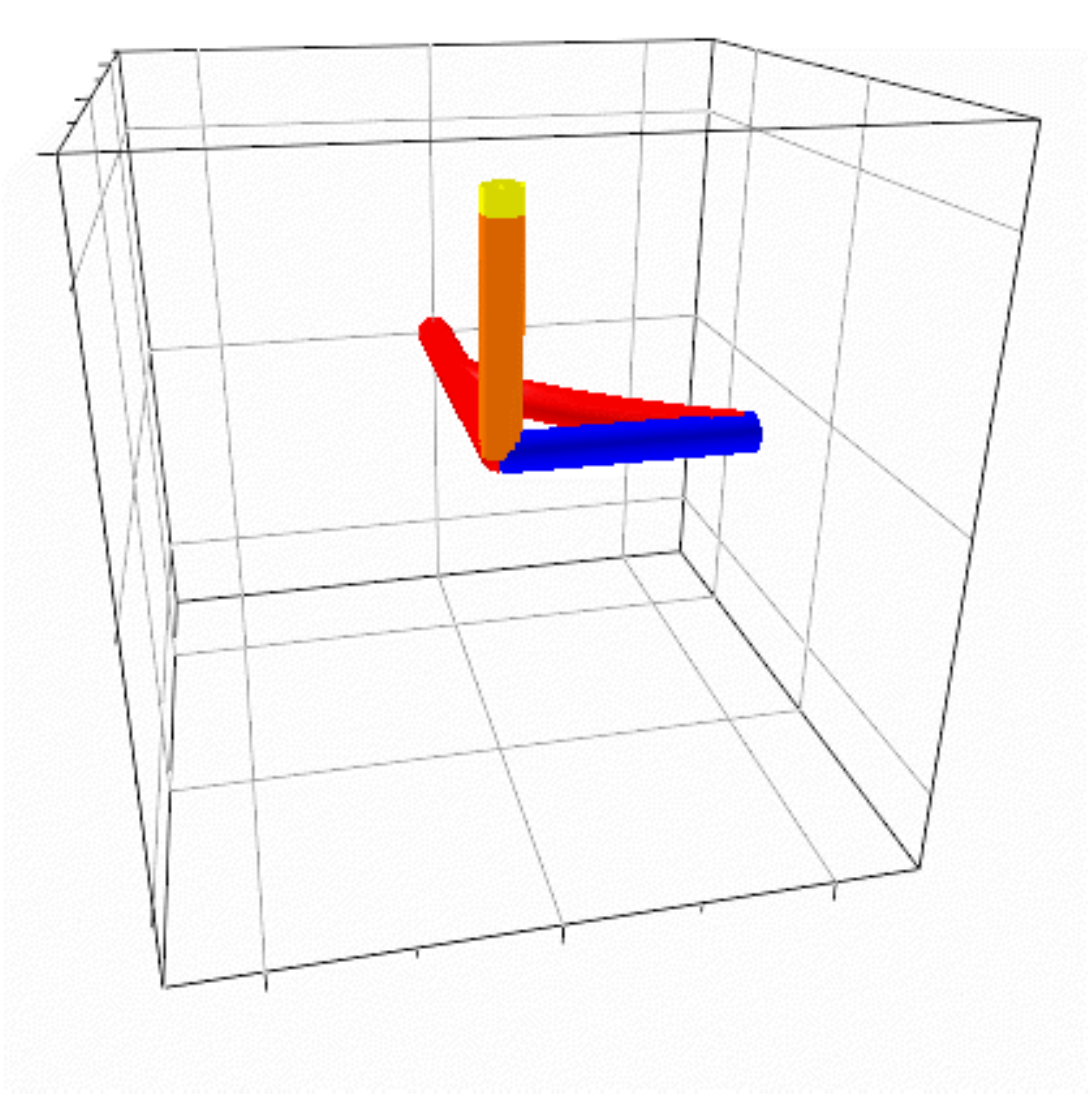}%
}
\\
\end{tabular}%
\end{table}%
%

\begin{table}[htbp] \centering
\begin{tabular}
[c]{ccccc}%
{\includegraphics[
natheight=4.166700in,
natwidth=4.147000in,
height=0.9447in,
width=0.9414in
]%
{hand-4-0.pdf}%
}
&
{\includegraphics[
natheight=4.375100in,
natwidth=4.354300in,
height=0.9461in,
width=0.9426in
]%
{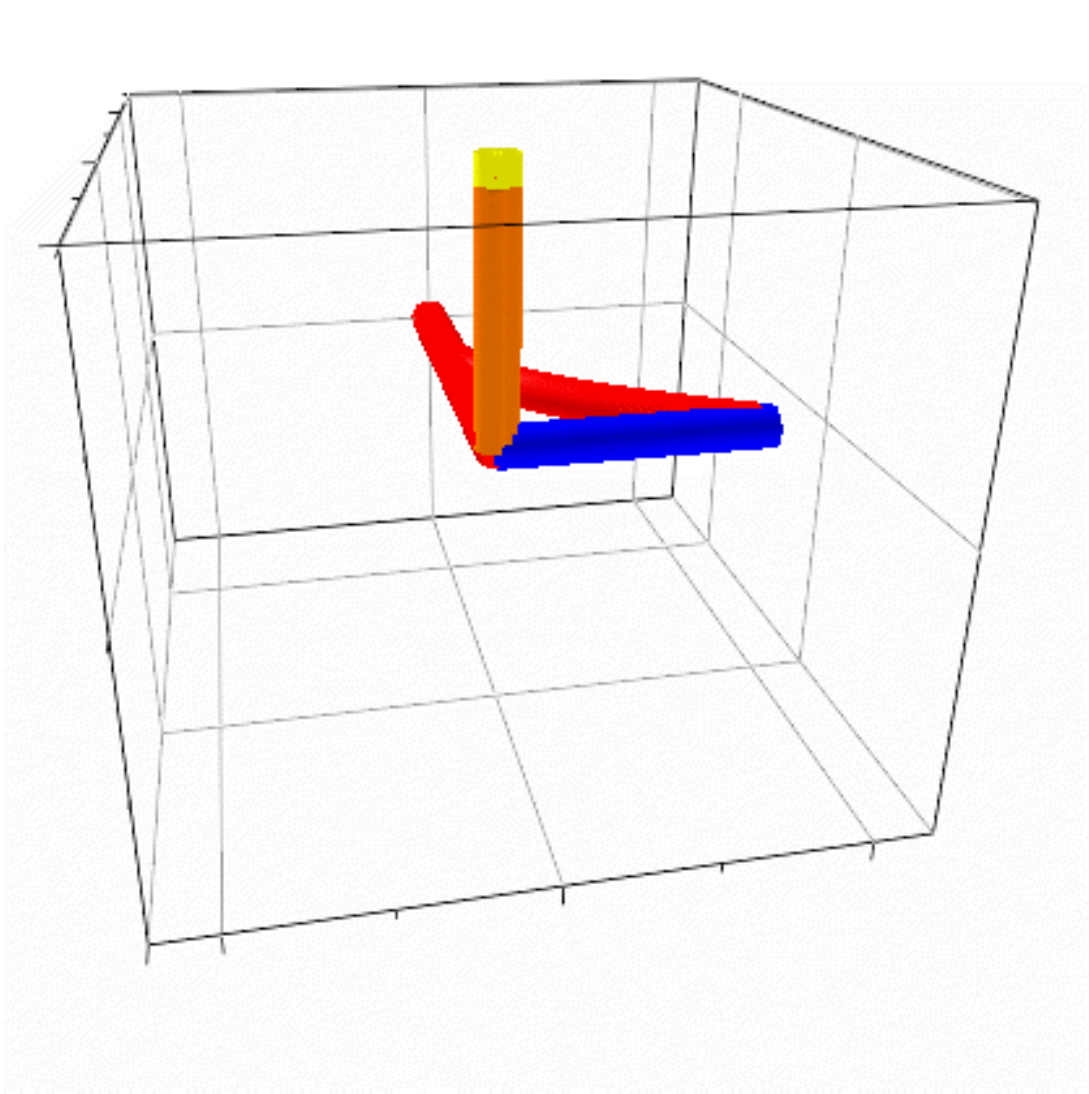}%
}
&
{\includegraphics[
natheight=4.375100in,
natwidth=4.354300in,
height=0.9461in,
width=0.9426in
]%
{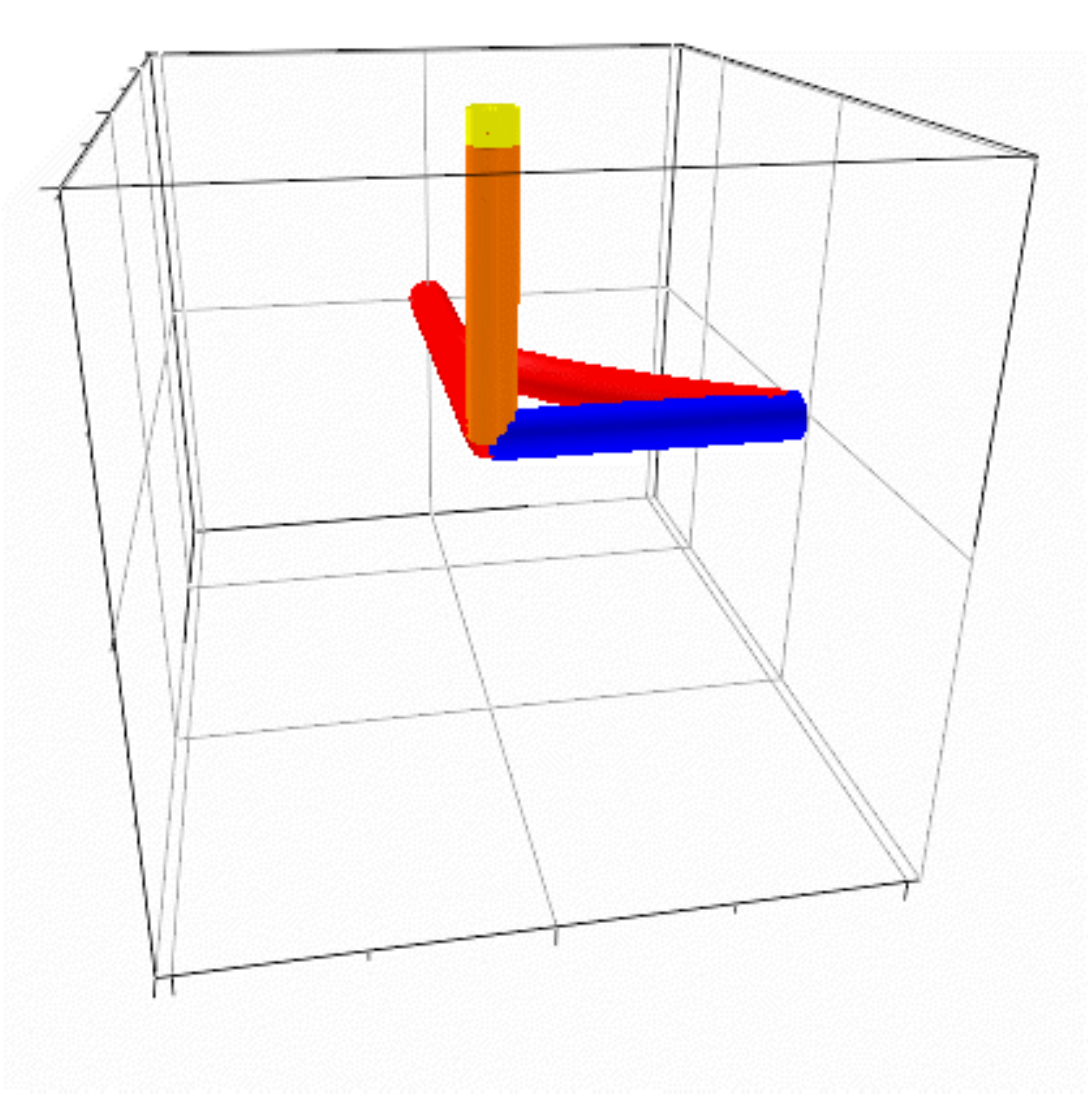}%
}
&
{\includegraphics[
natheight=4.375100in,
natwidth=4.354300in,
height=0.9461in,
width=0.9426in
]%
{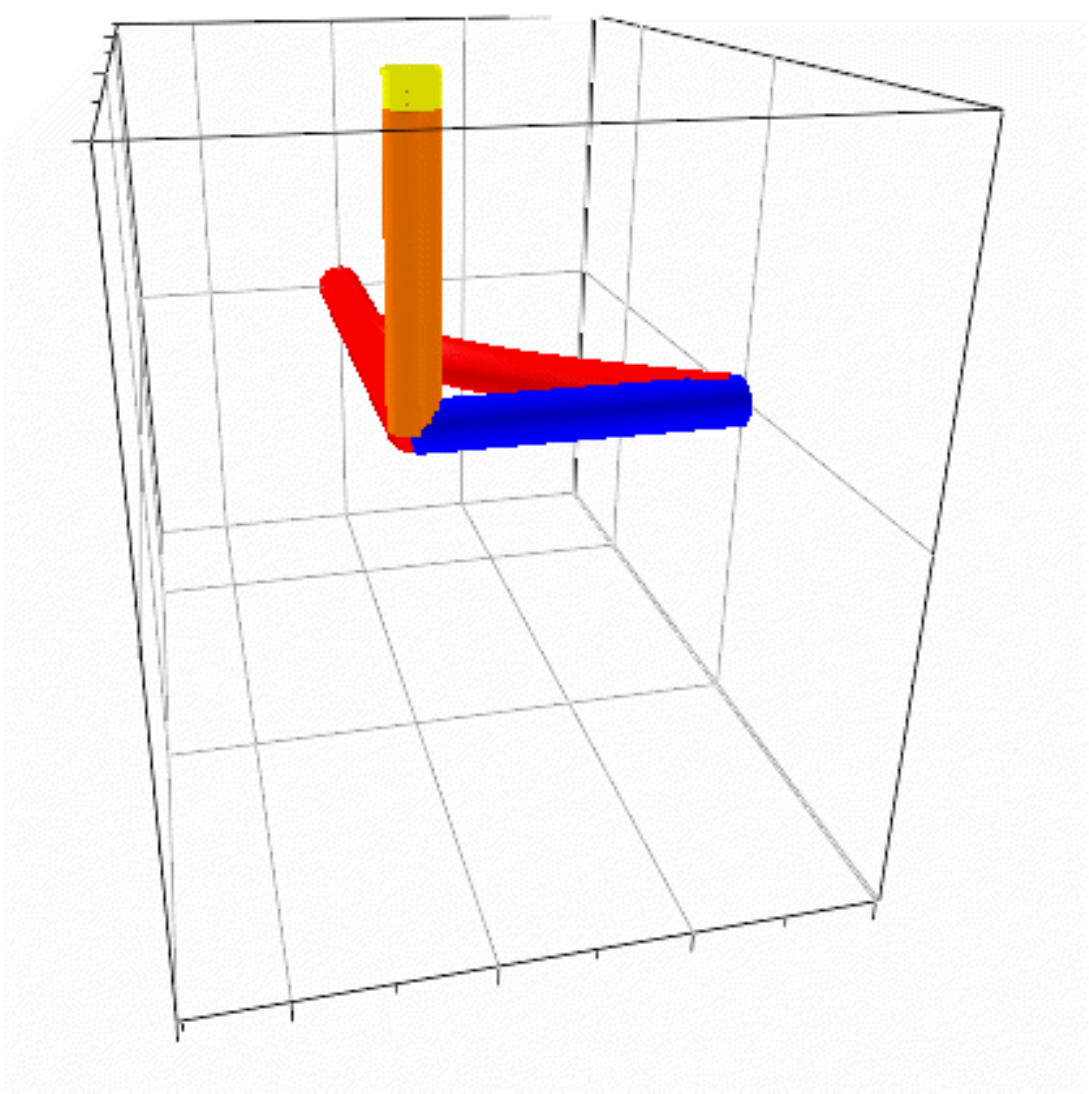}%
}
&
{\includegraphics[
natheight=4.375100in,
natwidth=4.354300in,
height=0.9461in,
width=0.9426in
]%
{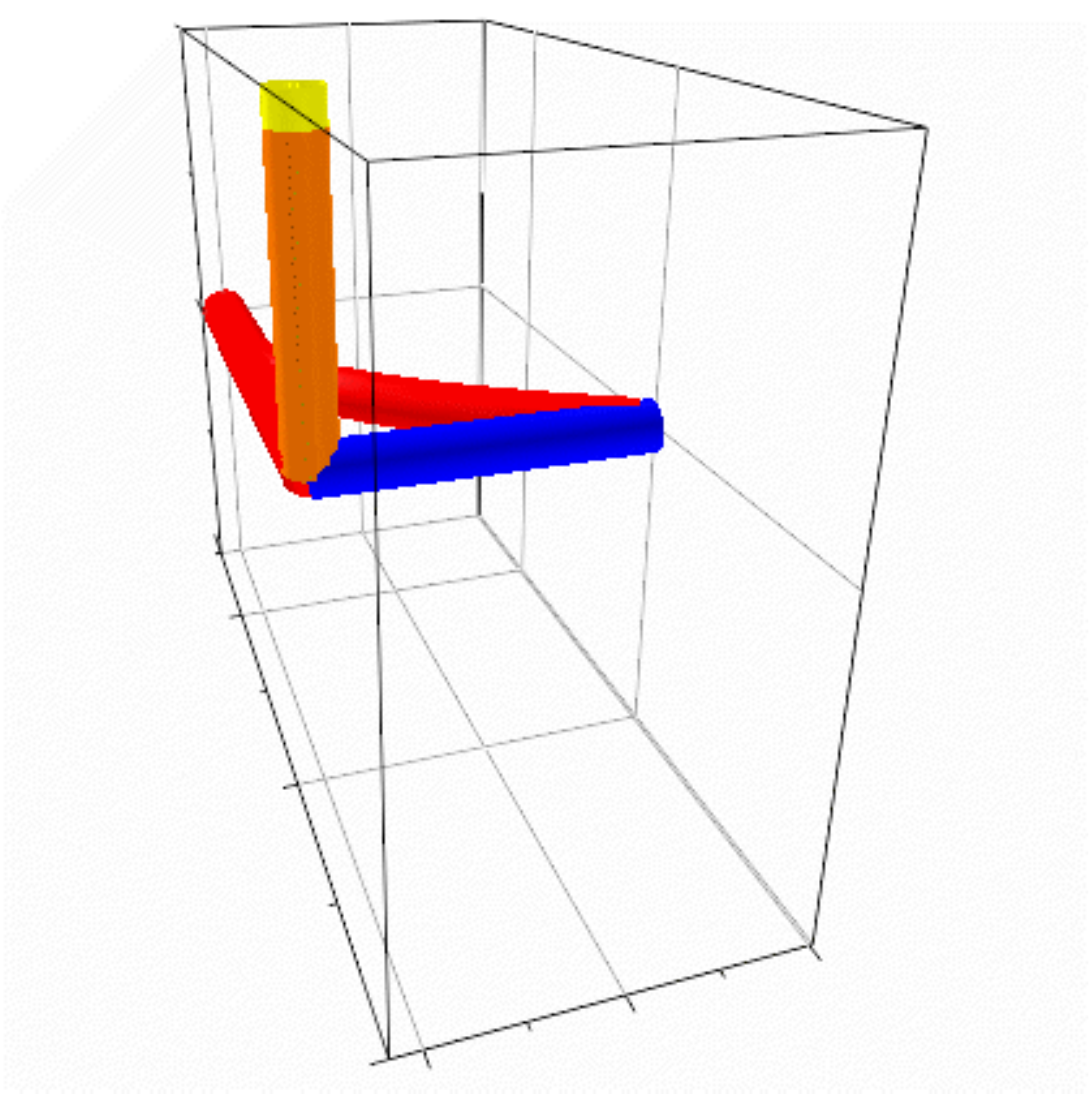}%
}
\\
\end{tabular}%
\end{table}%

The grid of frozen images uses eight increments for each of $s$ and $t$.
Moving down in the grid increases $t$, so each column displays images from the
movie that corresponds to a particular value of $s$. Moving rightwards
increases $s$, so the columns begin with the double rotation movie and end
with the stationary movie. The grid is split across two pages, with the column
for $s=\pi/4$, which lies in the middle of the nullhomotopy, repeated so as to
show on both pages.

Each movie displays a twirling \textquotedblleft right hand\textquotedblright,
chosen for its asymmetry, so that the hand's embedding in space corresponds to
a unique rotation. The hand has a long red finger(s), a blue thumb, connective
tissue between them, and an orange candle with burning yellow flame, attached
securely\footnote{As we have learned, the candle will turn upside down at some
point during the homotopy!} to the palm, in the spirit of the candle dance.
Every movie begins and ends at the same default embedding of the hand, in
which the fingers point directly away from the viewer (to $-I$), the candle
upright (to $K$), and the thumb to the right (to $J$). At $s=0$ we can watch
the hand rotate twice keeping the candle upright, tracing out the double
rotation around the vertical axis. As $s$ progresses the movies slowly
trivialize. The movies also show a smaller unmoving hand, in the default
embedding, for fixed visual reference. A moving green ray displays the axis of
rotation $\varphi$, and a green arc displays the rotation angle amount
$\theta$, reflecting the graphs of Figure \ref{phi-theta-graphs}.

Below are some features to look for in the movies, echoing results above,
which will be relevant in the next section when we teach you how to perform
the movies with your own body. These beautiful and fascinating features are
best seen using the MEGL program mentioned above, which has an optional
contrail feature, allowing one to view moving contrails that follow
individually the paths traversed by either the fingers, thumb, or candle, each
of which exhibits dramatically different behavior.

\begin{itemize}
\item The fingers should point directly away (to $-I$) often (since $-I$ is
its own hinge) and in every other direction exactly once. To achieve this the
contrail path followed by the fingers starts as a double cover of the equator,
which unfolds to a figure eight. The progressing figure eights, which always
self-intersect at $-I$, eventually shrink to nothing.

\item The thumb should point in every direction other than $J$ exactly once,
except to its antipodal point, where it should point often, because its
antipodal point $-J$ is its hinge. In fact, from Figure \ref{phi-theta-graphs}
we can see that for $s<\pi/4$, the thumb points to $-J$ twice during the
movie, then for $s=\pi/4$ it happens only once (with the thumb barely grazing
$-J$), right in the middle of the movie, and thereafter not at all. The viewer
may verify this in the frozen images and more easily in the animations. The
reader may even wish to verify from the graphs in Figure
\ref{phi-theta-graphs} that $s=\pi/4$ is the last movie in which any vector is
rotated to its antipodal point. To achieve this behavior, the contrail path
followed by the thumb again starts as a double cover of the equator and
unfolds to a figure eight. One loop of the progressing figure eights, which
always self-intersect at $-J$, contracts to form a cusp, and subsequently the
remaining loop shrinks to nothing at $J$.

\item The candle must point in every non-upright direction exactly once during
the interior of the homotopy (since $K$ is its own hinge); for instance, by
referring to the frozen image grid, you can pinpoint exactly when the candle
is upside down. To achieve this overall behavior, the contrail path followed
by the candle starts as a tiny cardiod, which wraps around the sphere as the
nullhomotopy progresses, eventually contracting to a cusp through
thin-petioled leaves. The cardiod-like behavior can be deduced from the graphs
in Figure \ref{phi-theta-graphs}.
\end{itemize}

Since the movies interpolate between a double rotation movie and an entirely
stationary movie, one could hope that each successive movie is a
\textquotedblleft relaxed\textquotedblright\ version of the one before. This
is nicely seen by viewing the computer animations. The viewer must mentally
interpolate the intermediate movies connecting our chosen values of $s$, since
to show everything would be an animation of animations, requiring one more
time coordinate than we currently have on hand, so to speak. Once you are
convinced that each animation looks like a relaxed\ version of the previous
one, then the mental homotopy interpolation is a success. This will prepare
you to perform the entire phenomenon with your body.

\section{You can \textquotedblleft wave\textquotedblright\ the untangling
yourself!}

In this section, we will teach you how to perform the motions depicted in our
movies with your hand. We refer to the hand motions for each movie as a
\emph{wave}. One might worry that our arms are not sufficiently flexible to
control the required hand motions, but the most physically demanding wave is
the first one in the homotopy (namely a double rotation of the hand about the
vertical) and this is already performed by the candle dance.

Recall that when we first described the candle dance, we observed that the
movement of the arm contains all the information for an untangling, but that
information is hidden in the intricate motions of our wrist, elbow, and
shoulder joints, and seems inaccessible. Our goal is to display the untangling
entirely with hand motions, as shown in our animations, with the arm now
serving only to control the hand. Thus the candle dance serves only to
demonstrate the first movie in our animations of $D$ ($s=0$). We should hope
that since the nullhomotopy relaxes a double rotation to no rotation at all,
the subsequent movies will become easier to perform.

At this point the reader should try performing the sequence of movies
appearing in our animations \texttt{doubletip-nullhomotopy.mp4} or in the
displayed grid of frozen images. Our best advice for learning how to wave the
movies is to pay close attention to the embedding of your hand halfway through
each movie. This should allow the rest of each movie to flow naturally. In the
video \texttt{waving.mp4} from our supplementary materials (also at
\cite{animations}) we lead you through this learning process, or you may
consult the following paragraph. And in the video
\texttt{simultaneous-waving.mp4} we demonstrate all the waving movies in synchrony.

For instance, consider the second movie in our series, immediately following
the double rotation. If you perform the double rotation (candle dance) without
sufficient limberness of arm, you will probably not manage to keep your palm
perfectly horizontal, especially in the middle of the movie, when this
requires the greatest contortion. Our second movie displays just such a
relaxed candle dance: The hand falls an eighth of a turn short of horizontal
in the middle of the movie. Our advice, then, is to use the fact that in the
middle of each wave, your fingers should point directly away from you, and the
embeddings of your hand at the middle of the successive movies should rotate
once clockwise (from your point of view) around the axis pointing directly
toward you. That is, in the $i$-th wave ($0\leq i\leq8$), aim for your hand to
be positioned, halfway through the wave, with the fingers pointing away and
your hand rotated $i/8$ of a turn clockwise around the axis pointing directly
toward you (as measured from its initial palm-up position with thumb pointing
to the right). This rotational positioning of the hand halfway through the
waves can be seen by looking across the middle row of the frozen images grid
(and can be gleaned from the graphs in Figure \ref{phi-theta-graphs}).

To perfect your motions, match them with the behavior of $-I$, $J$, $K$ (the
reference positions of your fingers, thumb, and imaginary candle) discussed in
the previous section. Good luck learning how to wave a nullhomotopy of the
double-twist, and have fun showing it to others!

\bigskip

\noindent\textbf{ACKNOWLEDGEMENTS.} We thank Ted Stanford for stimulating our
start on this journey, Bill Julian for suggesting the quaternionic viewpoint
on rotations, John MacKendrick and George Pearson of MacKichan Software for
help with animations using Scientific Workplace, and Marty Flashman for
suggesting a simultaneous waving video. We thank Robert Lipshitz for designing
a Macintosh app displaying our animations. And we thank Sean Lawton and the
students at the George Mason University Experimental Geometry Lab for their
animations of our nullhomotopy \cite{MEGLab}.


\begin{thebibliography}{99}                                                                                               %


\bibitem {conway-smith}John H. Conway, Derek A. Smith, \emph{On Quaternions
and Octonions: Their Geometry, Arithmetic, and Symmetry}, AK Peters, Natick,
Mass., 2003.

\bibitem {curtis}Charles Curtis, \emph{Linear Algebra}, Springer, New York, 1993.

\bibitem {delatorre}de la Torre, Antonio Martos, \emph{Dirac's Belt Trick for
Spin 1/2 Particle}, \url{http://vimeo.com/62228139}, accessed October 2, 2016.

\bibitem {egan}Greg Egan, \emph{Dirac Belt Trick, }%
\url{http://www.gregegan.net/APPLETS/21/21.html}, accessed October 2, 2016.

\bibitem {euler}Leonhard Euler, Formulae generales pro translatione quacunque
corporum rigidorum (General formulas for the translation of arbitrary rigid
bodies), E478, \emph{Novi Commentarii Academiae Scientiarum Petropolitanae}
\textbf{20}\emph{ }(1776), 189--207; \url{http://eulerarchive.maa.org/},
accessed October 2, 2016.

\bibitem {euler-wikipedia}\emph{Euler's Rotation Theorem}, Wikipedia,
\url{http://en.wikipedia.org/wiki/Euler's\_rotation\_theorem}, accessed
October 2, 2016.

\bibitem {euler-citizendium}\emph{Euler's Theorem (Rotation)}, Citizendium,
\url{http://en.citizendium.org/wiki/Euler's\_theorem\_(rotation)}, accessed
October 2, 2016.

\bibitem {Feynman}Richard P. Feynman and Steven Weinberg, \emph{Elementary
Particles and the Laws of Physics: The 1986 Dirac Memorial Lectures; Lecture
Notes Compiled by Richard MacKenzie and Paul Doust}, Cambridge University
Press, New York, Cambridge, 1987.

\bibitem {FK}George Francis, Louis Kauffman, Air on the Dirac strings, in
\emph{Mathematical Legacy of Wilhelm Magnus, }Contemporary Mathematics
\textbf{169} (1994), 261--276.

\bibitem {kauffman-video}George Francis et al, \emph{Air on the Dirac
Strings}, movie, University of Illinois, Chicago, 1993,
\url{http://www.evl.uic.edu/hypercomplex/html/dirac.html},
\url{http://www.evl.uic.edu/hypercomplex/movies/dirac.mpg},
\url{https://www.youtube.com/watch?v=CYBqIRM8GiY}, accessed October 2, 2016.

\bibitem {gray}Alfred Gray, Elsa Abbena, Simon Salamon, \emph{Modern
Differential Geometry of Curves and Surfaces with Mathematica}, Chapman \&
Hall CRC, Boca Raton, 2006.

\bibitem {gallier}Jean Gallier, \emph{Geometric Methods and Applications: For
Computer Science and Engineering}, Springer, New York, 2001.

\bibitem {hanson}Andrew J. Hanson, \emph{Visualizing Quaternions}, Elsevier, 2006.

\bibitem {HFK}John Hart, George Francis, Louis Kauffman, Visualizing
quaternion rotation, \emph{ACM Transactions on Graphics} \textbf{13} (1994), 256--276.

\bibitem {hirsch}Morris Hirsch, \emph{Differential Topology}, Springer, New
York, 1997.

\bibitem {misner}Charles W. Misner, Kip S. Thorne, John Archibald Wheeler,
\emph{Gravitation}, W.H. Freeman, San Francisco, 1973.

\bibitem {wikipedia-orientation-entanglement}\emph{Orientation Entanglement},
Wikipedia, \url{http://en.wikipedia.org/wiki/Orientation\_entanglement},
accessed October 2, 2016.

\bibitem {palais}Robert A. (Bob) Palais, \emph{Bob Palais' Belt Trick, Plate
Trick, Tangle Trick Links},\emph{ }\newline%
\url{http://www.math.utah.edu/~palais/links.html},\newline%
\url{http://www.math.utah.edu/~palais/JavaBeltPlateQuat/belt.html},\newline%
\url{http://www.math.utah.edu/~palais/bp.html}, \newline accessed October 2, 2016.

\bibitem {animations}David Pengelley, Daniel Ramras, \emph{The double-tipping
nullhomotopy,} \url{http://math.iupui.edu/~dramras/double-tip.html}, accessed
October 2, 2016. Also at \url{https://www.youtube.com/playlist?list=PLAfnEXvHU52ldJaOye-8kZV_C1CjxGx2C}.

\bibitem {Penrose}Roger Penrose, Wolfgang Rindler, \emph{Spinors and
Space-Time: Volume 1}, Cambridge University Press, 1984.

\bibitem {MEGLab}Donnelly Phillips, Mae Markowski, Joseph Frias, Mark Boahen,
\emph{Virtual Reality Experience of Nullhomotopy in }$SO(3,R)$\emph{, Oculus
Rift App}, George Mason University Experimental Geometry Lab,
\url{http://meglab.wikidot.com/visualization}, accessed October 2, 2016.

\bibitem {wikipedia-plate-trick}\emph{Plate Trick}, Wikipedia,
\url{http://en.wikipedia.org/wiki/Plate\_trick}, accessed October 2, 2016.

\bibitem {quaternions-spatial-rotation}\emph{Quaternions and Spatial
Rotation}, Wikipedia,
\url{http://en.wikipedia.org/wiki/Quaternions\_and\_spatial\_rotation},
accessed October 2, 2016.

\bibitem {indonesian-candle-dance-youtube}\emph{Tari Piring - Indonesian
Candle Dance, }\url{https://www.youtube.com/watch?v=0FQsisO1gOY}, accessed
October 2, 2016.
\end{thebibliography}
\end{document}